\documentclass[11pt]{article}
\usepackage[english]{babel}
\usepackage{amsmath,amsthm,amsfonts,amssymb,bbm}
\usepackage{mathrsfs}
\usepackage[]{natbib} 
\usepackage[pagebackref=true,linktoc=page,colorlinks=true,linkcolor=blue,citecolor=blue]{hyperref}
\usepackage{hyperref}
\usepackage[retainorgcmds]{IEEEtrantools} 
\usepackage{enumerate}

\usepackage{enumitem}
\usepackage{makecell}
\usepackage{multirow}
\usepackage{graphicx}
\usepackage{here}
\usepackage[all]{xy}
\usepackage{tabularx} 
\usepackage{booktabs} 
\usepackage{accents} 
\usepackage[font=small,labelfont=bf]{caption} 
\usepackage{dsfont}
\usepackage{pstool,psfrag}

\newtheorem{Lemma}{Lemma}
\newtheorem{Theorem}[Lemma]{Theorem}
\newtheorem{Corollary}[Lemma]{Corollary}

\newtheorem{Proposition}[Lemma]{Proposition}
\newtheorem{Remark}[Lemma]{Remark}

\newcommand{\R}{\mathbb{R}}

\newcommand{\N}{\mathbb{N}}

\renewcommand{\P}{\mathbb{P}}

\newcommand{\E}{\mathbb{E}}
\renewcommand{\1}{\mathbbm{1}}

\DeclareMathOperator*{\argmax}{arg\,max}

\renewcommand{\d}{\mathrm{d}}
\newcommand{\ind}{\mathbf{1}}
\newcommand{\supp}{\mathop{\mathrm{supp}}}

\newcommand{\MSE}{\mathop{\mathrm{MSE}}}
\newcommand{\MISE}{\mathop{\mathrm{MISE}}}
\newcommand{\MRV}{\mathop{\mathrm{MRV}}}
\numberwithin{equation}{section}
\begin{document}

\title{Conditionally Max-stable Random Fields\\ based on log Gaussian Cox Processes}

\author{M.~Dirrler\footnote{Universit\"at Mannheim, A5,6 68161 Mannheim, Germany, Email address: mdirrler@mail.uni-mannheim.de},\, 
M.~Schlather\footnote{Universit\"at Mannheim, A5,6 68161 Mannheim, Germany, Email address: schlather@math.uni-mannheim.de},\, 
and K.~Strokorb\footnote{Universit\"at Mannheim, A5,6 68161 Mannheim, Germany, Email address: strokorb@math.uni-mannheim.de}}

\maketitle
\thispagestyle{empty}

\begin{abstract}
	{We introduce a class of spatial stochastic processes in the max-domain of attraction of familiar max-stable processes. The new class is based on Cox processes and comprises models with short range dependence. We show that statistical inference is possible within the given framework, at least under some reasonable restrictions.}
\end{abstract}

{\small
  \noindent \textit{Keywords}: {conditionally max-stable process; Cox extremal process; log Gaussian Cox process; mixed moving maxima; non-parametric intensity estimation
  }\\
  \noindent \textit{2010 MSC}: {Primary 60G70, 60G55} \\
  \phantom{\textit{2010 MSC}:} {Secondary 60G60} }


\section{Introduction}\label{sec:intro}

Probabilistic modelling of spatial extremal events is often based on the assumption that daily observations lie in the max-domain of attraction of a max-stable random field which justifies statistical inference by means of block maxima procedures. This methodology is applied in various branches of environmental sciences, for instance, heavy precipitation \citep{cooley2005}, extreme wind speads \citep{engelke2015huessler,genton2015maxstable,oesting2015postprocessing} and forest fire danger \citep{stephenson2015fire}.

At the same time modelling extreme observations on a smaller time scale is a much more intricate issue and to date only few non-trivial processes are known to lie in the max-domain of attraction (MDA) of familiar max-stable processes. Among them $\alpha$-stable processes \citep{samorodnitsky94} form a natural class which may be rich enough to cover a wide range of environmental sample path behaviour  \citep{stoev05alphastable} and scale mixtures of Gaussian processes with regularly varying scale, are known to lie in the domain of attraction of extremal t-processes \citep{opitz2013}.
However, stable processes are themselves complicated objects whose statistical inference is a challenging research topic \citep{nolan16book} and scale mixtures of Gaussian processes have an unnatural degree of long-range dependence. 

Our objective in this article is to introduce another class of spatial processes in the MDA of familiar max-stable models, which encompasses processes with short-range dependence and to explore whether statistical inference on them is feasible, at least under some reasonable restrictions. 

It is well-understood that max-stable processes can be built from Poisson point processes \citep{dehaan1984,ginehahnvatan1990maxstable,stoevtaqqu2006fractbr}. 
In order to define our new class of models,
we modify the underlying Poisson point process such that its intensity function is no longer fixed, but may depend on some spatial random effects. We pursue this idea by introducing conditionally max-stable processes based on Cox processes \citep{Cox55} which naturally generalize the class of mixed moving maxima processes \citep{smith1990maxstable,schlather2002modelsmaxstable,zhang04,stoev08}. Section~\ref{sec:model} contains the definition of our proposed model. A functional convergence theorem shows that these processes lie in the MDA of familiar mixed moving maxima processes. From a practical point of view, we choose to model the spatial random effects that influence the intensity function by a log Gaussian random field which makes the theory and application of log Gaussian Cox processes conveniently available for our setting, cf.\ \cite{moeller1998lgcp, moeller2004ppbook, moeller2010thinningpp, diggle2013loggaussiancox}.

For statistical inference we take a closer look at two important aspects in the recovery of the underlying Gaussian process. In Section~\ref{sec:estpsinonpara}, we deal with non-parametric estimation of realizations of the random intensity of the Cox process from observations of the conditionally max-stable processes and their storm centres. Based on the outcome of this procedure, we consider parametric estimation of the covariance function of the Gaussian process in Section~\ref{sec:estPsipara}. The performance of these procedures is examined in Section \ref{sec:simulationstudy} in a simulation study. To this end, we discuss in Section~\ref{sec:simu} how exact simulation of our proposed model can be traced back to exact simulation of max-stable random fields as in \cite{schlather2002modelsmaxstable}. Tables, proofs and auxiliary results are postponed to the Appendix~\ref{sec:appendix}.

%
%
%



\section{Model specification}\label{sec:model}
In this section we define our proposed model and study its properties. Since random measures and point processes are the building blocks for max-stable processes as well as our conditional model, we briefly review conventions based on \cite{daley2003pp, daley2008pp} in Appendix~\ref{sec:pointprocesses}.\\

\textbf{Reminder on Cox processes.}	The following point processes are relevant for our setting. A \textit{Poisson (point) process} $N\sim PP(\lambda)$ with directing measure $\lambda$  is a point process which satisfies that for any Borel set $B$ with $\lambda(B)<\infty$ the number of points $N(B)$ in the region $B$ is Poisson distributed with parameter $\lambda(B)$, i.e. $N(B)\sim poi(\lambda(B))$ and, secondly, that the random variables $N(B_1),\dots,N(B_n)$ are jointly independent for disjoint Borel sets $B_1,\dots,B_n$ in $E$.
We say $N$ is a \textit{Cox process} directed by the random measure $\Lambda$ and write $N \sim CP(\Lambda)$ if, conditional on $\Lambda$, it is a \textit{Poisson process}, i.e. $N|_{\Lambda=\lambda} \sim PP(\lambda)$.
By definition, \textit{Poisson processes} are special cases of \textit{Cox processes} with deterministic directing measure which coincides with their intensity measure, that is $\E N(B)=\lambda(B)$ for Borel sets $B$ and $N\sim PP(\lambda)$. For Cox processes, we have $\E N(B)=\E_{\Lambda}(\Lambda(B))$ instead.

The following lemma may be regarded as a central limit theorem for Cox processes and will be useful for our work. It means that the superposition of $n$ i.i.d.\ copies of an $1/n$-thinned Cox process converges weakly to a Poisson process and is an immediate consequence of Theorem~11.3.III in \cite{daley2008pp}.
\newpage
\begin{Lemma}
	\label{lemma:CP-CLT}
	Let $N^{(n)}_i \stackrel{i.i.d.}\sim CP(n^{-1}\Lambda)$, $i=1,\dots,n$ for $n=1,2,\dots$ be a triangular array of Cox processes. Then
	\begin{align*}
		\sum_{i=1}^n N^{(n)}_i \rightarrow PP(\lambda),\quad \text{where} \quad \lambda(A)=\E\Lambda(A),\quad \forall A \in \mathscr{B}(E).
	\end{align*}
\end{Lemma}


\textbf{Definition of Cox extremal processes.} Let $X$ be a (possibly deterministic) non-negative stochastic process on $\R^d$ that we call \emph{storm process} or \textit{shape}. We assume $X$ to have continuous sample paths. 
Based on its law $\P_X$ (on the complete separable metric space $\mathbb{X}=C(\R^d)$ with the usual Fr{\'e}chet metric) and another non-negative 
stochastic process $\Psi$ on $\R^d$, to be called \emph{spatial intensity process}, and a positive \emph{scaling constant} $\mu_Y$, we define a random field $Y$ on $\R^d$ by
\begin{align} \label{eq:defn-Y} 
	Y(t)= \bigvee_{i=1}^\infty u_i X_i(t-s_i), \qquad t \in \R^d,
\end{align}
where $N=\sum_{i=1}^\infty \delta_{(s_i,u_i,X_i)}$ is a Cox-process on $\mathbb{S}=\R^d\times (0,\infty]\times \mathbb{X}$, directed by the random measure 
\begin{align} \label{eq:Lambda}
	\mathrm{d}\Lambda(s,u,X)=\mu_Y^{-1} \, \Psi(s)\mathrm{d}s \, u^{-2} \mathrm{d}u \, \mathrm{d}\P_X. 
\end{align}
The randomness of the measure $\Lambda$ is due to the randomness of the spatial intensity process $\Psi$. 
Similarly to the situation with mixed moving maxima processes \citep{smith1990maxstable,zhang04,stoev08} or, more generally, extremal shot noise \citep{serra1984shotnoise1,serra1988shotnoise2,heinrich1994shotnoise,dombry2012shotnoise}, we will think of the processes $X_i$ as being random storms centred around $s_i$ that will affect its surroundings with severity $u_i$.  In case, the intensity process is almost surely identically one ($\Psi \equiv 1$), the construction of $Y$ is indeed the usual mixed moving maxima process
\begin{align} 
	\label{eq:defn-Z}
	Z(t)=\bigvee_{i=1}^\infty u_i X_i(t-s_i), \qquad t \in \R^d,
\end{align}
where $\sum_{i=1}^\infty \delta_{(s_i,u_i,X^{(i)})}$ is the Poisson process on $S$ with directing measure
\begin{align*}
	\d\lambda(s,u,X)= {\mu_Z}^{-1} \, \d s \, u^{-2} \d u \, \d \P_X.
\end{align*}
Note that, conditional on its intensity process $\Psi$,	the extremal process $Y$ is a (non-stationary) max-stable mixed moving maxima process.
In the sequel, we call $Y$ a \emph{conditionally max-stable random field} or \textit{Cox extremal process}.
\subsection{Properties of Cox extremal processes}\label{sec:Props} 
$\phantom{a}$\\
\textbf{Continuity, Stationarity and Max-Domain of Attraction.}
Even though the Cox extremal process $Y$ in \eqref{eq:defn-Y} itself is not max-stable, we show in this section that it lies in the max-domain of attraction of an associated mixed moving maxima random field $Z$ under rather general conditions. 
To show this, we first clarify some technical requirements that guarantee the finiteness and the continuity of sample paths of $Y$ and $Z$.


\begin{Lemma}[Finiteness and Sample-Continuity]\label{lemma:sample-cts}		
	Let $K$ be a (not necessarily compact) subset of $\R^d$.
	\begin{enumerate}
		\item If the integrability condition
		\begin{align} 
			\label{eq:int-cond-a}
			\E_\Psi \left(\E_X \left( \int_{\R^d} \sup_{t \in K} X(t-s) \Psi(s)~\d s\right)\right) <\infty 
		\end{align}	
		holds, then $\sup_{t \in K}Y(t)$ is almost surely finite. 
		\item If, additionally,  
		\begin{align}
			\label{eq:cond-sample-cts}
			\exists\, n \in \N \,:\, \inf_{t \in K} \bigvee_{i=1}^n u_i X_i(t-s_i) > 0 \quad \text{almost surely},
		\end{align}
		then the sample paths of the process $Y$ are almost surely continuous on $K$.	
		\item If both \eqref{eq:int-cond-a} and \eqref{eq:cond-sample-cts} are satisfied for any compact $K \subset \R^d$, 
		then  $Y$ is almost surely finite on compact sets and sample-continuous on $\R^d$. In this case only finitely many points of $N$ contribute to $Y$ on $K$. 
	\end{enumerate}
\end{Lemma}
\begin{Remark}
	The Cox extremal process $Y$ is in general not uniquely determined by the choice of its shape $X$ and intensity process $\Psi$. For instance, let $\tilde{X}$ be a process which satisfies the same assumptions as $X$, and independently of $\tilde{X}$, let $\xi$ be a random variable, such that $X$ can be decomposed into 
	\begin{align*}
		X(t)=\tilde{X}(t)\xi,\quad t\in\R^d.
	\end{align*}
	Then choosing $\tilde{X}$ as shape and $\Psi\cdot\xi$ as intensity process does not alter the finite dimensional marginal distributions of the Cox extremal process $Y$, since 
	\begin{align*}
		\P(Y(t_1)\leq y_1,\dots, Y(t_n)\leq y_n)=\E_{\Psi}\exp\left(-\E_X\int \max_{i=1,\dots,n} \frac{X(t_i-s)}{y_i} \Psi(s)~\d s\right).
	\end{align*}
\end{Remark}

In the sequel, we will always assume that the intensity process $\Psi$ has continuous sample paths, is strictly stationary and almost surely strictly positive with
\begin{align}
	\label{eq:int-cond-rho}
	c_\Psi = \E_\Psi\Psi(o) < \infty,
\end{align}
where $o\in\R^d$ denotes the origin.
These assumptions simplify some requirements of the preceding lemma.
For instance, by Tonelli's theorem, condition \eqref{eq:int-cond-a} will be equivalent to 
\begin{align}\label{eq:int-cond-X}
	\E_X \left( \int_{\R^d} \sup_{t \in K} X(t-s) \d s\right) <\infty.
\end{align}	
For $K=\{t\}$, we obtain that
\begin{align}
	\label{eq:int-cond-b}
	\E_\Psi \bigg( \E_X\int_{\R^d}X(t-s)\Psi(s)\d s \bigg) = c_\Psi \cdot \E_X \bigg(\int_{\R^d} X(s)~\d s \bigg)  <\infty
\end{align}	
entails the finiteness of $Y(t)$ as well as $Z(t)$ for $t \in \R^d$.

In fact, the mixed moving maxima field $Z$ in \eqref{eq:defn-Z} has standard Fr\'echet margins if its scaling constant ${\mu_Z}$ equals \eqref{eq:int-cond-b} with $\Psi\equiv 1$, that is $c_{\Psi}=1$.
Condition \eqref{eq:int-cond-a} will be trivially satisfied for compact subsets $K$ of $\R^d$ 
if $\Psi$ is stationary, $c_\Psi \in (0,\infty)$ and 
\begin{align}\label{cond:Xbounded}
	X \leq C \ind_{B_{R}(o)},\quad \P_X\text{-almost surely}
\end{align}
for some positive constants $C,R>0$.  Here, $B_R({{o}}) \subset \R^d$ denotes the closed ball of radius $R$ centred at $o\in\R^d$.
On the other hand, the following lemma gives a simpler requirement to ensure the somewhat cumbersome condition \eqref{eq:cond-sample-cts}. Note that $\inf_{s \in K} \Psi(s) >0$ is always satisfied for compact $K$ if we assume the (sample-continuous) intensity process $\Psi$ to be almost surely strictly positive. 

\begin{Lemma}\label{lemma:sample-cts2}
	If $K$ is compact, $\inf_{s \in K} \Psi(s) >0$ holds $\P_\Psi$-almost surely, and, additionally,
	\begin{align}
		\label{cond:Xsupport}
		\exists\, r>0 \text{ such that } \P_X\big(B_r(o)\subset \supp(X)\big)>0
	\end{align}
	with $\supp(X)=\{s \in \R^d \,:\, X(s)>0 \}$, then condition \eqref{eq:cond-sample-cts} holds true for $K$.
\end{Lemma}

\begin{Remark}
	In the definition of the Cox extremal process $Y$ it is also possible to work with storm processes $X$ that may attain negative values, such as Gaussian processes. If at least \eqref{cond:Xsupport} is satisfied and the sample-continuous intensity process $\Psi$ is almost surely strictly positive, the resulting random field $Y$ will be almost surely strictly positive.
\end{Remark}

\begin{Lemma}[Stationarity]\label{lemma: stationarity}
	If the intensity process $\Psi$ is stationary, then the Cox extremal process $Y$ is stationary.
\end{Lemma}

Finally, we formulate the main result of this section.

\begin{Theorem}[Max-Domain of Attraction]
	\label{thm:domainofattraction}
	Let the (sample-continuous) intensity process $\Psi$ be stationary and almost surely strictly positive satisfying \eqref{eq:int-cond-rho} and let the (sample-continuous) storm process $X$ satisfy conditions \eqref{eq:int-cond-X} and \eqref{cond:Xsupport}. Then the random fields $Y$ and $Z$ are finite on compact sets and sample-continuous, and  the random field $Y$ lies in the max-domain of attraction of $Z$. More precisely, if the scaling constant $\mu_Y$ equals the integral \eqref{eq:int-cond-b} and $\mu_Z=\mu_Y/c_{\Psi}$, then the following convergence holds weakly in $C(\R^d)$
	\begin{align*}
		n^{-1}\bigvee_{i=1}^n Y_i\rightarrow Z,
	\end{align*}	
	where $Y_i$ are i.i.d.\ copies of $Y$.
\end{Theorem}

\textbf{Choices for the intensity process.} For inference reasons we shall further assume henceforth that the intensity process $\Psi$ is a stationary log Gaussian random field, that is,
\begin{align*}
	\Psi(s)=\exp(W(s)), \qquad s \in \R^d,
\end{align*}
where $W$ is stationary and Gaussian. Thereby, all requirements for $\Psi$ from the preceding Theorem~\ref{thm:domainofattraction} are guaranteed as long as $W$ has continuous sample paths. Moreover, the latter also ensures that the distribution of the random measure $\Lambda$, cf.\ \eqref{eq:Lambda}, is uniquely determined by the distribution of $W$. By \cite{moeller1998lgcp} (see also \cite{adler1981geometry}, page 60), a Gaussian process
$W$ is indeed sample-continuous if its covariance function $C$ satisfies $1-C(h)<M \lVert h \rVert^{\alpha},$ $h\in\R^d$, for some $M>0$ and $\alpha>0$.
This condition holds for most common correlation functions, for instance, the stable model $C(h)=\exp(-\lVert h \rVert^{\alpha}),$ $\alpha \in (0,2], h\in\R^d$, and the Whittle-Mat\'ern model 
\begin{align}\label{eq:matern}
	C(h)=\frac{2^{1-\nu}}{\Gamma(\nu)}(\sqrt{2\nu}h)^{\nu}K_{\nu}(\sqrt{2\nu}h),\quad \nu>0, h\in\R^d,
\end{align}
see \cite{gneiting2006matern}.\\

\textbf{Choices for the storm profiles.} For statistical inference, we rely on identifying at least some of the centres of the storms $X_i$ from observations of $Y$. As a starting point, it is reasonable to assume that the paths of $X$ satisfy a monotonicity condition, for instance that for each path $X_{\omega}$ there exist some monotonously decreasing functions $f_{\omega}$ and $g_{\omega}$ such that
\begin{align}\label{eq:Xmonotone}
	g_{\omega}(\|t\|)\leq X_{\omega}(t)\leq f_{\omega}(\|t\|)
\end{align}
and $g_{\omega}(0)= X_{\omega}(0)= f_{\omega}(0)$.
For the purpose of illustration, we will use in most of our examples a deterministic shape $X=\varphi$, with $\varphi$ being the density of the $d$-dimensional standard normal distribution as in \cite{smith1990maxstable}. See also Section~\ref{sec:discussion} for a discussion of this choice and the recovery of storm centres.


\section{Simulation}\label{sec:simu}

In many cases, functionals of max-stable processes cannot be explicitly calculated, e.g., for most models only the bivariate marginal distributions are known while the higher dimensional distributions do not have a closed-form expression. Therefore and in order to test estimation procedures, efficient and sufficiently exact simulation algorithms are desirable. However, exact simulation of (conditionally) max-stable random fields can be challenging, since a priori, its series representation \eqref{eq:defn-Y} involves taking maxima over infinitely many storm processes.
A first approach in order to simulate mixed moving maxima processes and some other max-stable processes was presented in \cite{schlather2002modelsmaxstable}. Meanwhile, several improvements with respect to exactness and efficiency have been proposed in \cite{engelke11,oesting12simbr,oesting13normspect,dieker2015simu,dombry2016} and \cite{liu2016simu}, some of which are mainly concerned with the simulation of Brown-Resnick processes.
%

Since our focus in this work is not on the simulation algorithm, it will be sufficient for us to extend the straight forward approach of \cite{schlather2002modelsmaxstable} in this article.

%

Under the (mild) conditions of Lemma~\ref{lemma:sample-cts} only finitely many
of the storms in \eqref{eq:defn-Y} contribute to the maximum if we restrict the random field to a
compact domain $D \subset \R^d$, see also \cite{deHaan2006}. Still, the centres of these contributing storms could be located on the whole $\R^d$. In order to define a feasible and exact algorithm we consider bounded storm profiles $X$ with bounded support, i.e. which satisfy condition \eqref{cond:Xbounded}. In such a situation only storms with centres within the enlarged region
$$D_R=D\oplus B_R(o)=\bigcup_{s\in D}B_R(s),$$
can contribute to the maximum \eqref{eq:defn-Y}. 

%
\begin{Proposition}[Simple Simulation Algorithm]\label{Prop:Simu}
	Let $D \subset \R^d$ be a compact subset and assume that the conditions of Lemma~\ref{lemma:sample-cts2} hold true and additionally the storm profile  X  satisfies almost surely \eqref{cond:Xbounded}. Then the following construction leads to an exact simulation algorithm on $D$ for the associated Cox extremal process $Y$.
	\begin{itemize}
		\item Let $\psi$ be a realization of the intensity process $\Psi$ and 
		$\nu_\psi(\cdot)~=~\int_{\cdot}\psi(s)~\d s$
		the associated measure on $\R^d$.
		\item Let $S_i \stackrel{i.i.d.}\sim \psi/\nu_\psi(D_R)$, $i=1,2,\dots$ be an i.i.d.\ sequence of random variables from the probability measure $\psi/\nu_\psi(D_R)$ on $D_R$.
		\item Let $X_i \stackrel{i.i.d.}\sim X$, $i=1,2,\dots$ be an i.i.d.\ sequence of storm profiles.
		\item Let $\xi_i$, $i=1,2,\dots$ be an i.i.d.\ sequence of standard exponentially distributed random variables and set $\Gamma_n=\sum_{i=1}^n \xi_i$ for $n=1,2,\dots$. 
	\end{itemize}
	Based on the stopping time
	\begin{align*}
		T=\inf\bigg\{n\geq 1: \Gamma_{n+1}^{-1}C \leq \inf_{t \in D} \bigvee_{i=1}^n \Gamma_i^{-1}X_i(t-S_i)\bigg\},
	\end{align*}
	we define the random field $\widetilde{Y}$ on $D$ via
	\begin{align*}
		\widetilde{Y}(t)=\frac{\nu_\psi(D_R)}{\mu_Y} \bigvee_{i=1}^{T} \Gamma_i^{-1}X_i(t-S_i),\quad t\in D.
	\end{align*}
	In this situation the following holds true.
	\begin{enumerate}
		\item The stopping time $T$ is almost surely finite.
		\item The law of the process $\widetilde{Y}$ coincides with the law of the Cox extremal process $Y$ restricted to $D$.
	\end{enumerate}	
\end{Proposition}
Beyond this extension, we would like to point out that in fact all previous procedures for simulation of instationary mixed moving maxima processes can be adapted for the simulation of Cox extremal processes in a similar way. For instance, by using a transformed representation of the original process $Y$, the efficiency improvement of \cite{oesting13normspect} can be transferred as well.

\begin{Remark}
	When condition \eqref{cond:Xbounded} is not satisfied, we choose $R$ and $C$ such that
	\begin{align}\label{eq:approxsim}
		\P\bigg(\sup_{t\in \R^d\setminus B_R(o)}X(t)>\varepsilon\bigg)\leq \alpha \quad \text{ and } \quad 	\P\bigg(\sup_{t\in B_R(o)}X(t)>C\bigg)\leq \alpha
	\end{align}
	hold true for some prescribed small $\varepsilon>0$ and $\alpha>0$ and approximate $X$ by its truncation $\min(X\1_{B_R},C)$ in the preceding algorithm, whence simulation will be only approximately exact. 
	For example, let us consider a generalization of the Smith model in $\R^d$, i.e. $X=\varphi$ with $\varphi$ the bivariate standard normal density. Then \eqref{eq:approxsim} is satisfied with $\varepsilon=10^{-4}, \alpha=0, R\approx 3.89$ and $C=(2\pi)^{-1/2}$. Figure \ref{fig:Plot1a} depicts two plots of a Cox extremal process $Y$ and its underlying log Gaussian random field $\Psi$.
\end{Remark}

%
\begin{figure}
	\centering
	\begin{minipage}[b]{7.7 cm}
		\includegraphics[width=0.8\textwidth]{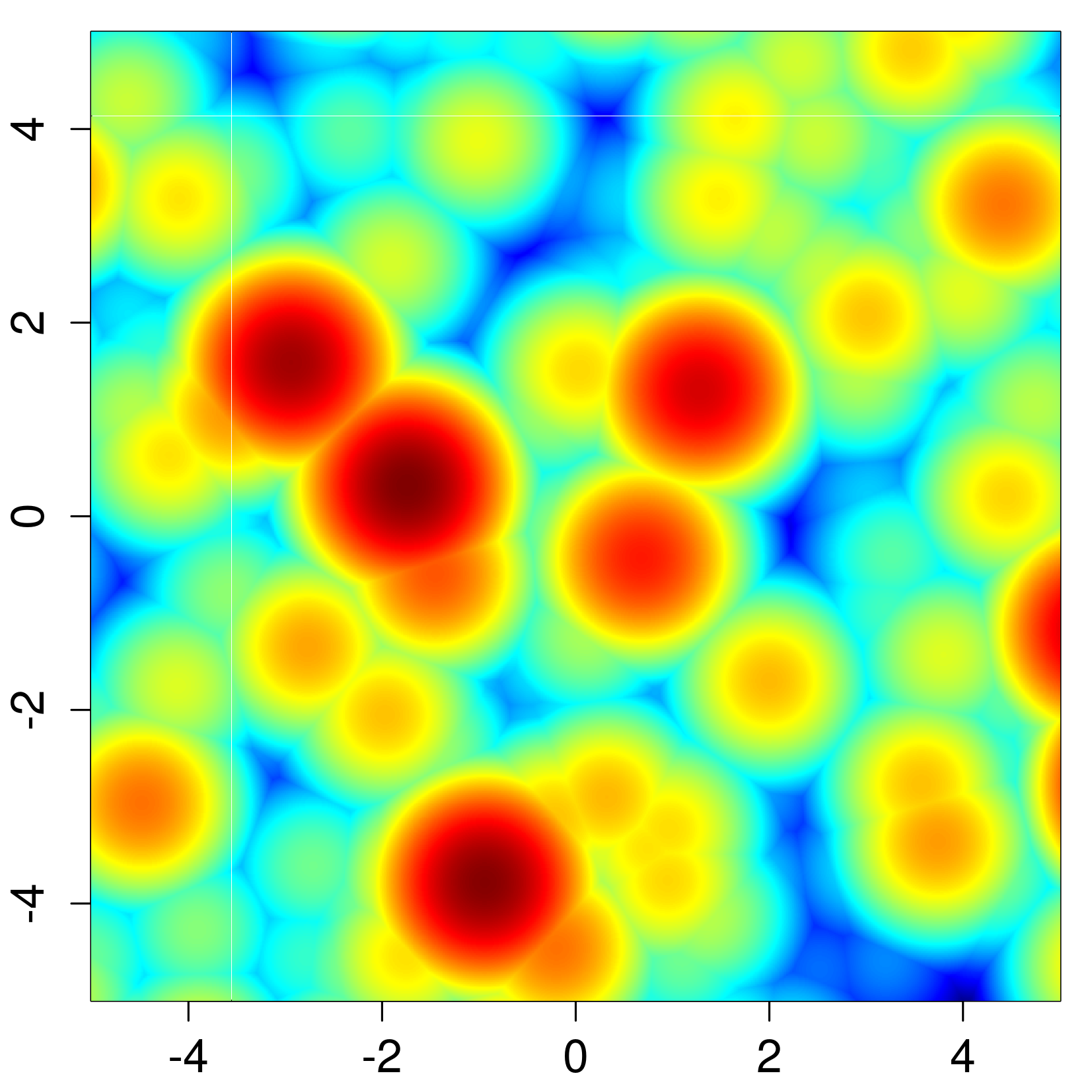}
		\includegraphics[width=0.17\textwidth]{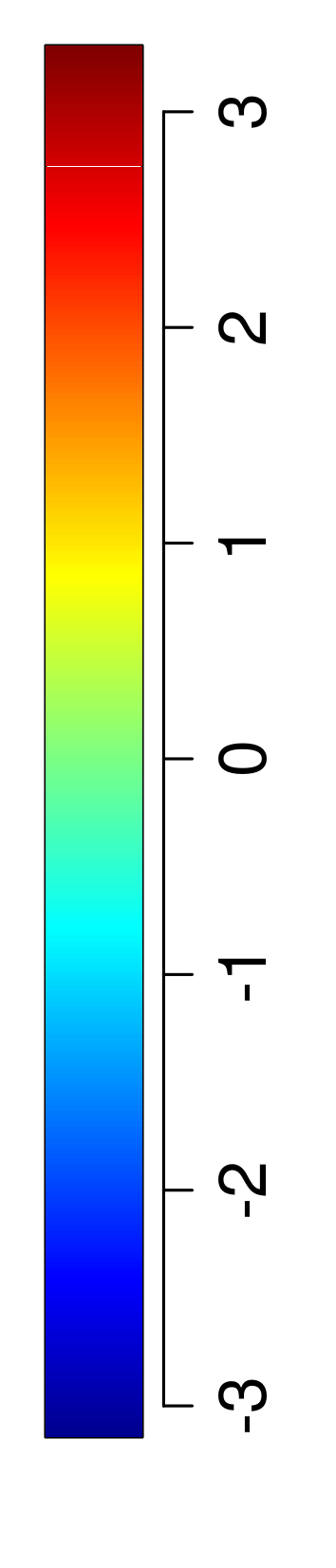}
	\end{minipage}
	\begin{minipage}[b]{7.7cm}
		\includegraphics[width=0.8\textwidth]{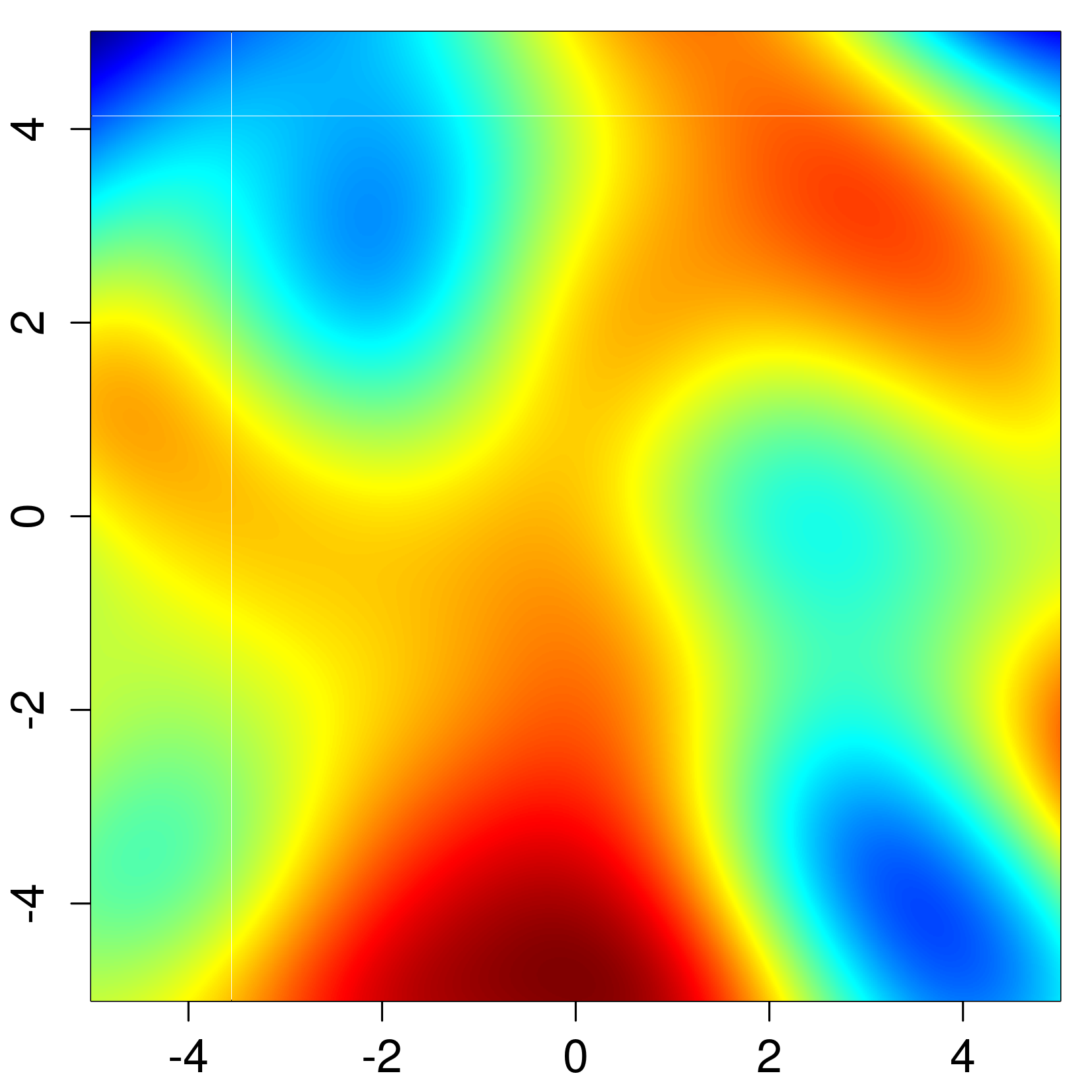}
		\includegraphics[width=0.17\textwidth]{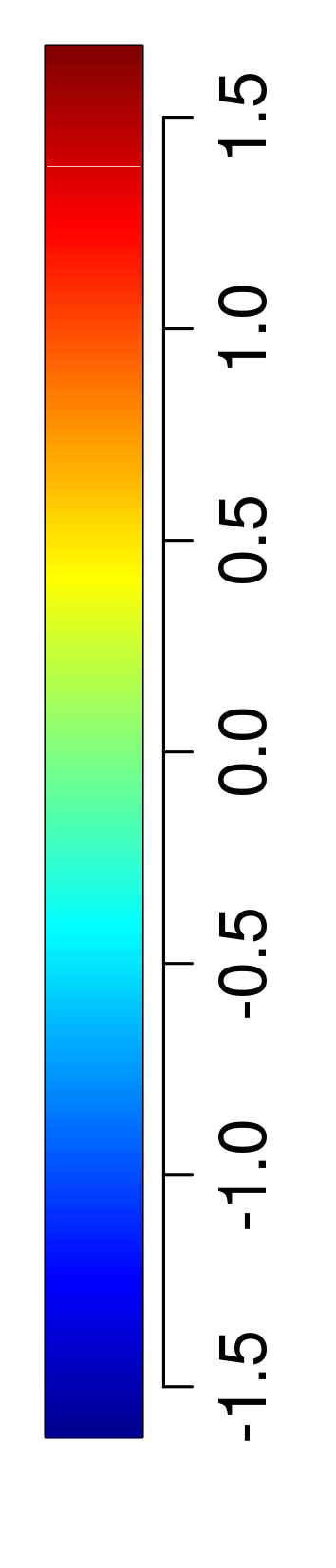}
	\end{minipage}

	\begin{minipage}[b]{7.7cm}
		\includegraphics[width=0.8\textwidth]{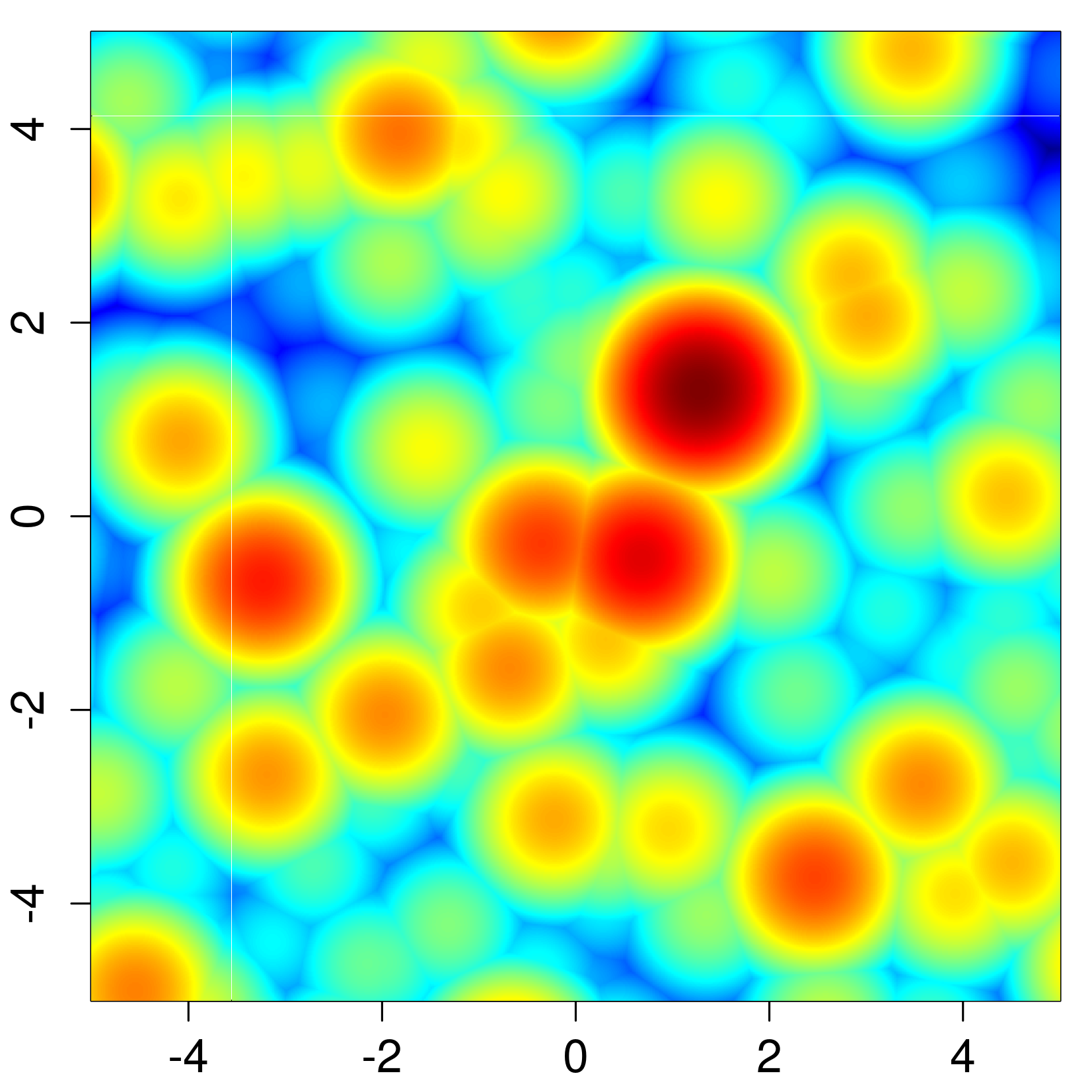}
		\includegraphics[width=0.17\textwidth]{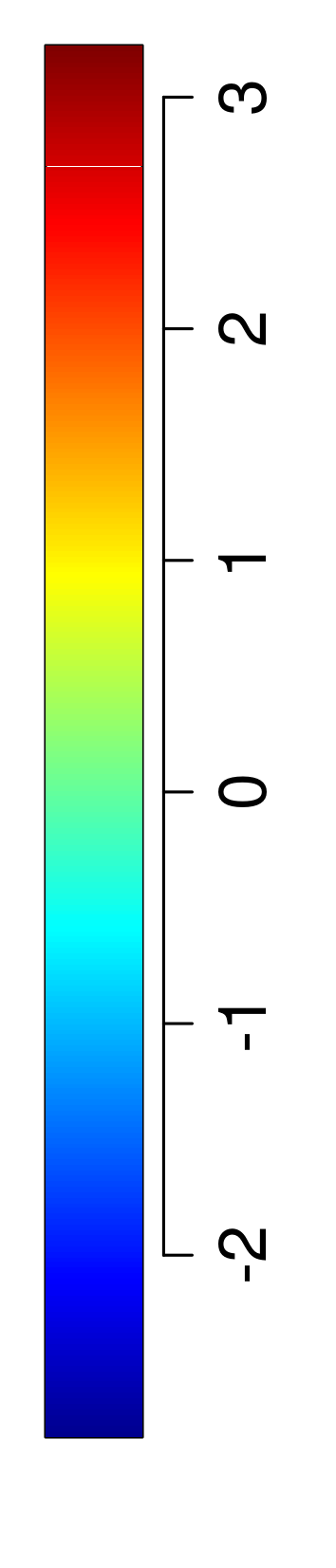}
	\end{minipage}
	\begin{minipage}[b]{7.7cm}
		\includegraphics[width=0.8\textwidth]{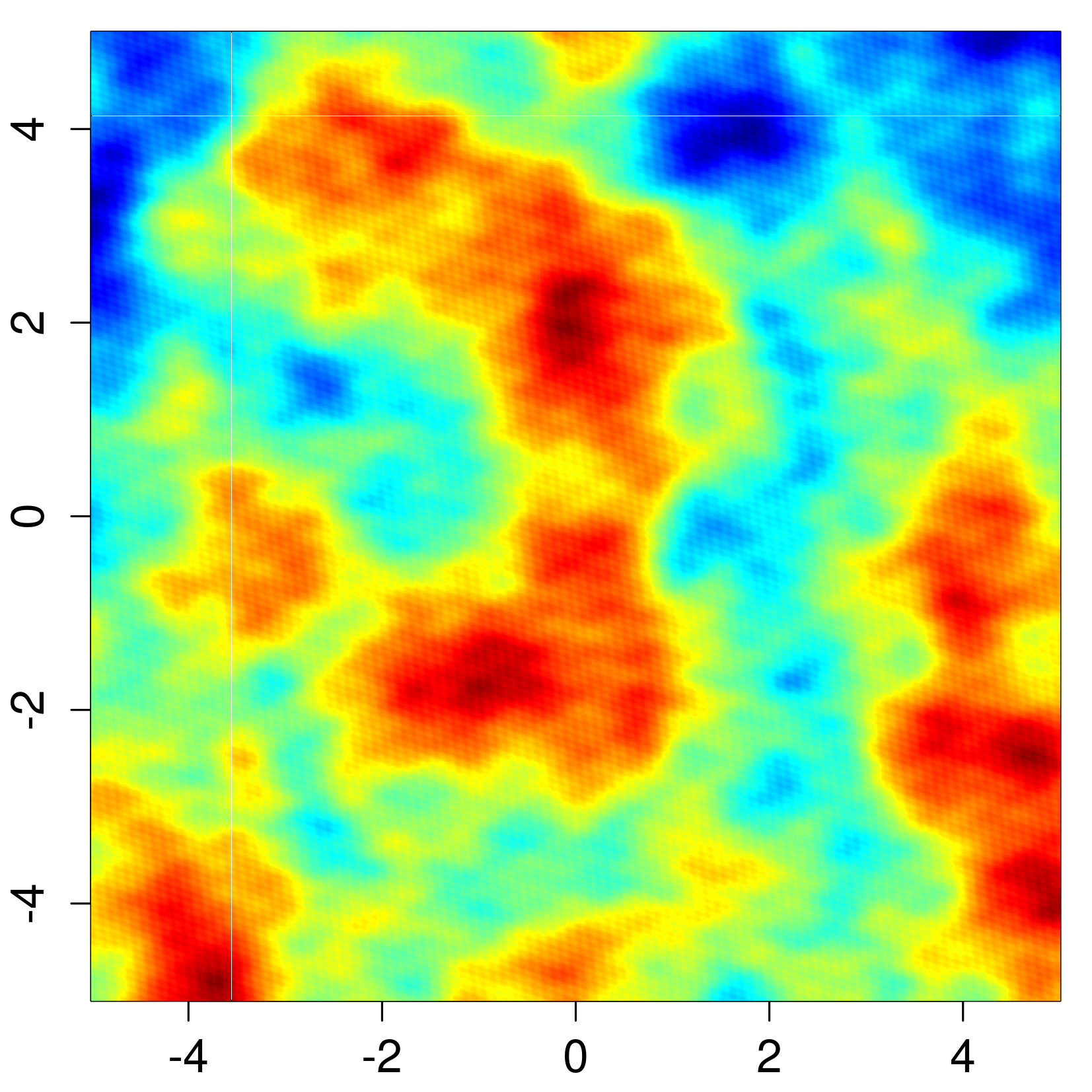}
		\includegraphics[width=0.17\textwidth]{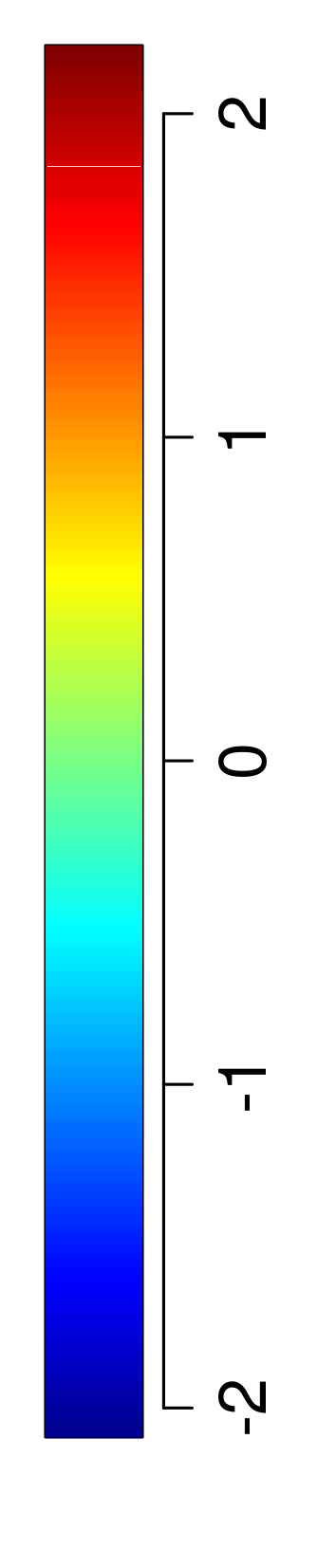}
	\end{minipage}
	\caption{Cox extremal processes $Y$ (left) and underlying log Gaussian random fields $\Psi$ (right).
		The covariance of $\log \Psi$ is of Whittle-Mat\'ern type with $\mathrm{var}=1$, $\mathrm{scale}=2$ and $\nu=\infty$ (upper plots) and $\nu=1$ (lower plots) respectively. The plots have been transformed to a logarithmic scale and the storm profiles have deterministic shape $X=\varphi$.}
	\label{fig:Plot1a}
\end{figure}
\section{Non-parametric inference on the realization $\psi$ of the intensity process $\Psi$}\label{sec:estpsinonpara}
The first part of this section provides theoretical tools for inference on the realization $\psi$ of the spatial intensity process $\Psi$. We derive a non-parametric estimator of a single realization $\psi$ from observations of the Cox extremal process $Y$ and their storm centres and state conditions for its convergence. In the second part we adapt this estimator to a non-asymptotic setting.

In both situations we assume that the shape function $X$ is known. Estimation of $X$ is beyond the scope of this article. However, since the process $Y$ lies in the MDA of $Z$, procedures for parametric estimation of the shape of the mixed moving maxima process $Z$ itself can be applied. For instance \textit{madograms} (see \cite{matheron1987} and \cite{cooley2005}), \textit{censored likelihood} \citep{nadarajah1998mex,schlathertawn2003depmeasure} or \textit{composite likelihood} \citep{castruccio15}. The recent article of \cite{huser2016} gives an overview of such likelihood methods.
We focus on inference of the spatial intensity process $\Psi$ hereinafter.	

\subsection{General theory and asymptotics} \label{sec:estpsitheo}

We are aiming at recovering the  realization $\psi$ of the intensity process $\Psi$ from observations of the Cox extremal process $Y$ as in \eqref{eq:defn-Y} and the ensemble of contributing storm centres $s_i$.
To this end, we define the point process
\begin{align}
	N_K^{Y^{\ast}}=\sum_{i=1}^{\infty}\delta_{s_i}\1_{\{\sup_{t\in K} X_i(t-s_i)Y^{\ast}(t)^{-1}\geq u_i^{-1}\}},
\end{align}
for a compact set $K\subset\R^d$ and some almost surely positive random field $Y^{\ast}$ which can be interpreted as a data-driven threshold and which we allow to depend on $Y$.
Indeed, if $Y^{\ast}$ equals $Y$, then the resulting process $N_K^Y$ is the ensemble of locations whose corresponding shape functions contribute to $Y$ on $K$, see also Figures \ref{fig:contstorms} and \ref{fig:contstormscenterv2}. That is, $N_K^{Y}$ is the location component of the process of extremal functions introduced by \cite{dombry2013contpoints} and \cite{oesting2014randpartition}.

Moreover, conditions \eqref{cond:Xbounded} and \eqref{cond:Xsupport} imply that $N_K^{Y}$ is almost surely a finite point process. Furthermore, with $K_R=K\oplus B_R(o)$ we have
\begin{align}\supp \left(N_K^{Y}\right)\subset K_R, \label{eq:contpoints}	\end{align}
almost surely. 
\begin{figure}
	\centering
	\begin{minipage}[b]{7.7cm}
		\includegraphics[width=1\textwidth]{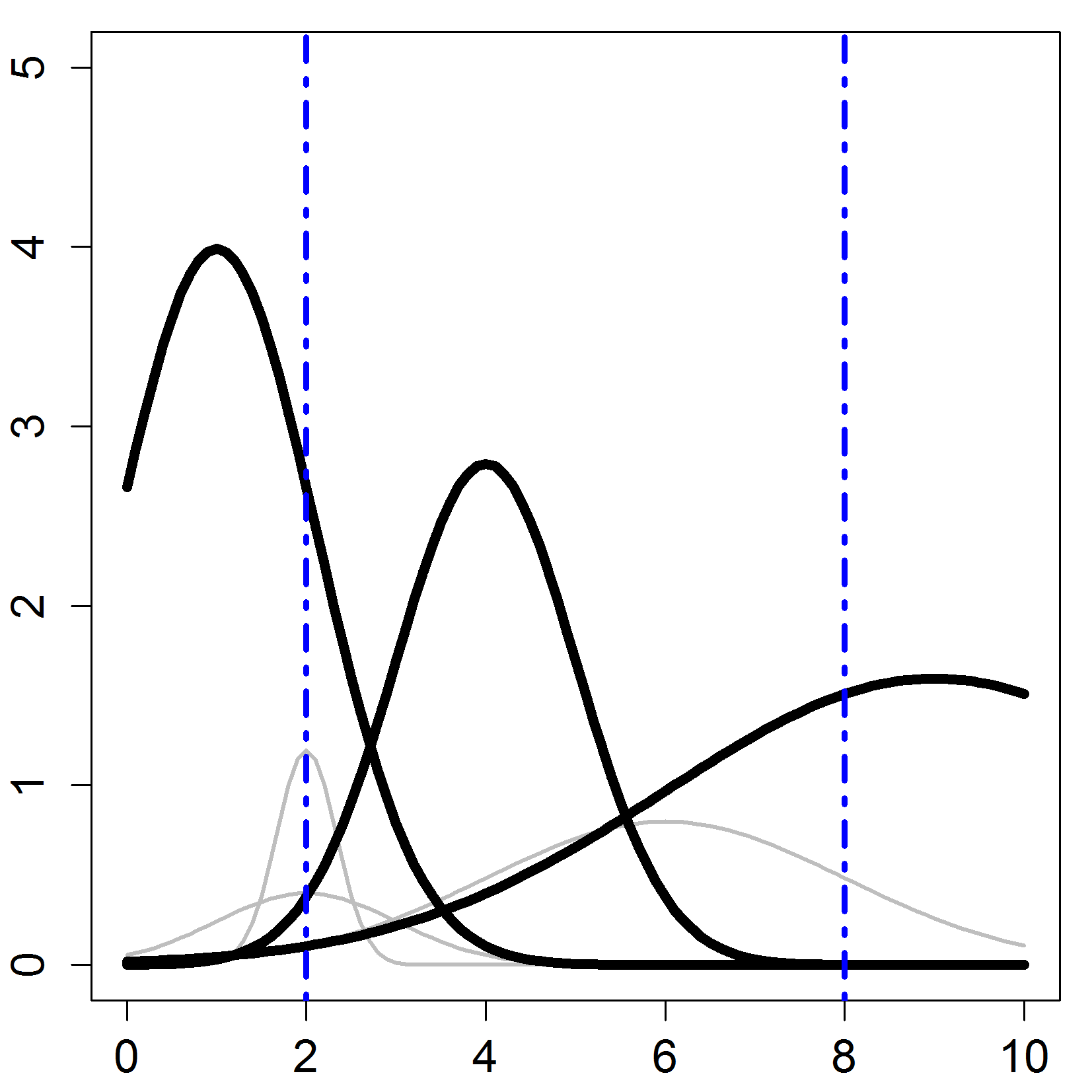}
		\caption{Contributing storms (black), the grey ones do not contribute to the final process on $K~=~[2,8]$.\protect\phantom{Some points of $N_K^Y$ are outside of $K$.aaaaaaaaaaaaaaaaaaaaaaaaaaaaaaaaaaaaaaaaaa}\\}
		\label{fig:contstorms}
	\end{minipage}
	\hspace{6mm}
	\begin{minipage}[b]{7.7cm}
		\includegraphics[width=1\textwidth]{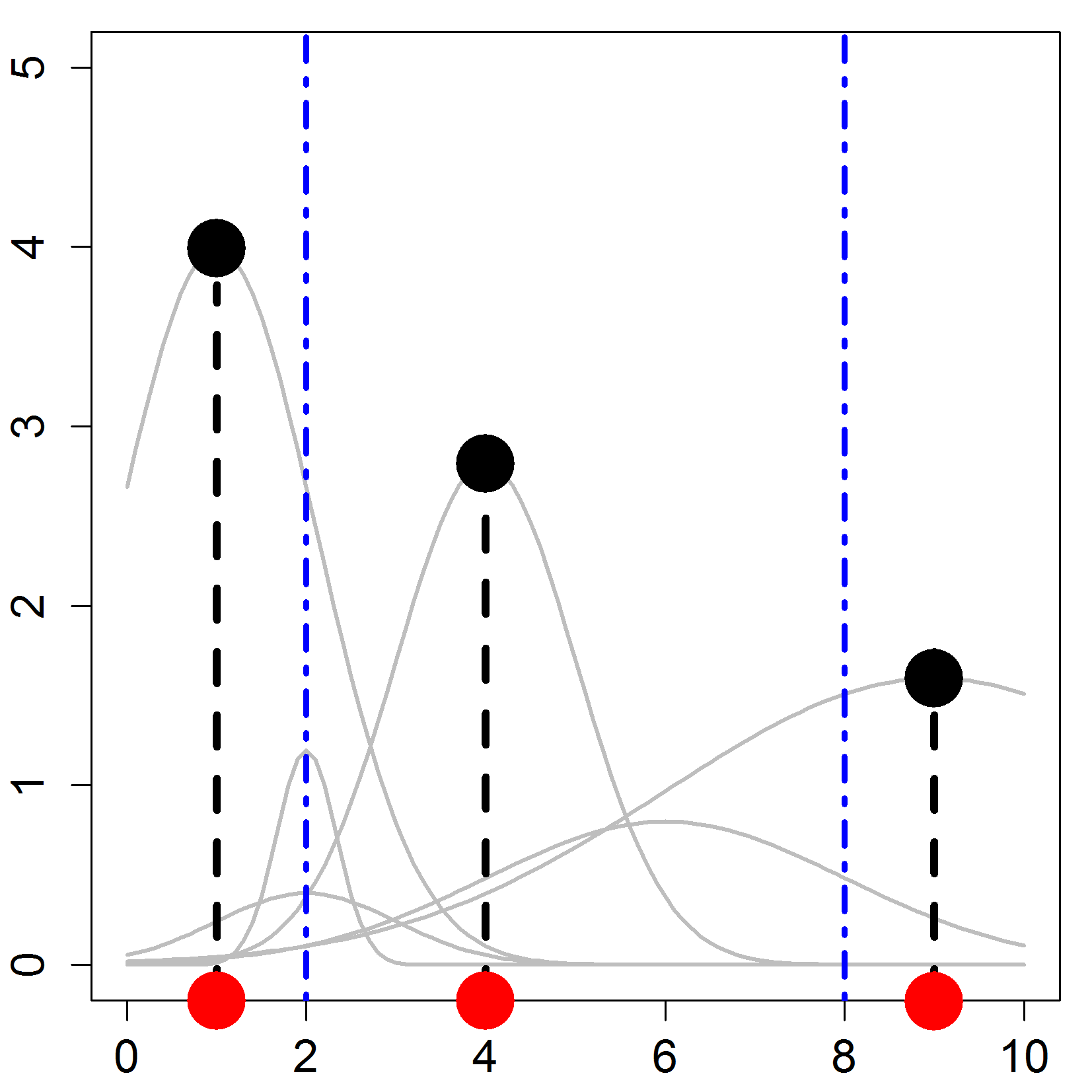}
		\caption{Storm centres (black dots) and location of the storms (red). The red dots correspond to the point process $N_K^{Y}$. Some points of $N_K^Y$ are outside of $K$.}
		\label{fig:contstormscenterv2}
	\end{minipage}
\end{figure}
To advance theory, we henceforth assume that, conditional on $\Psi=\psi$, the process $Y^{\ast}$ is a copy of $Y$ that is independent of $N$. 
\begin{Lemma}
	\label{lemma:contpointsint}
	Assume that the conditions \eqref{cond:Xbounded} and \eqref{cond:Xsupport} are satisfied. Let $Y^{\ast}\vert_{\Psi=\psi}$ be an independent copy of $Y\vert_{\Psi=\psi}$ which is independent of $N \vert_{\Psi=\psi}$. Let $K\subset\R^d$ be a compact set. Then $N_K^{Y^{\ast}}$ is a Cox process on $K_R$. More specifically
	\begin{align*}
		N_K^{Y^{\ast}}\vert_{\Psi=\psi,Y^\ast=y} \sim PP\left( b_K^{y}(s) \psi(s)~\d s\right)
	\end{align*}
	is a Poisson point process, whose intensity function equals $\psi(s)$ up to the correcting factor
	\begin{align}\label{eq:by}
		b^{y}_K(s)=\mu_{Y}^{-1} \E_X\left[\sup_{t\in K} \frac{X(t-s)}{y(t)}\right], \quad s \in \R^d,
	\end{align}
	whose support lies in the set $K_R$.
\end{Lemma}
The correcting factor $b_K^{y}(s)$ is in principle known if the shape process $X$ is known. If $n$ i.i.d.\ realizations
\begin{align*}
	\varphi_i \stackrel{i.i.d.}{\sim}N_K^{Y^{\ast}} \vert_{\Psi=\psi,Y^*=y}
\end{align*}
of the point process $N_K^{Y^{\ast}} \vert_{\Psi=\psi,Y^*=y}$ are given, we estimate its intensity $\psi_K^y(s)=b^y_K(s)\psi(s)$ in a non-parametric way by a kernel estimator as follows. Consider some kernel $k(\cdot)$ and bandwidth $h_{m_n}$. Let $\sum_{i=1}^{m_n} \delta_{t_i}$ represent the ensemble of individual points of the point process $\sum_{i=1}^n \varphi_i$. 
A sequence of kernel estimators $\widehat{{\psi}}_{K,n}^{y}$ for $\psi^y_K$ is given by
\begin{align}\label{eq:kerneldensest}
	\widehat{{\psi}}_{K,n}^{y}(s)=\frac{1}{m_nh_{m_n}^d}\sum_{i=1}^{m_n}k\left(\frac{s-t_i}{h_{m_n}}\right).
\end{align}

\begin{Theorem}[Uniform Convergence, \cite{rao83unifconv}]\label{lemma:uniformconvergence}
	Assume that the sequence of bandwidths $h_n$ satisfies $h_n\rightarrow 0$ and $nh_n^d/|\log n|\rightarrow\infty$ and let $c_{K}^y=\int \psi_K^y(s)~\d s$. Then conditions \eqref{cond:Xbounded} and \eqref{cond:Xsupport} imply that
	\begin{align*}
		\underset{s\in\R^d}{\sup\ }|c_K^y\widehat{{\psi}}_{K,n}^{y}-{\psi}_K^{y}|\rightarrow 0,\quad \text{almost surely.}
	\end{align*}
\end{Theorem}
The estimator \eqref{eq:kerneldensest} for $\psi_K^y=b_K^y\psi$ suggests to estimate the intensity $\psi$ on the interior of $K_R$ by
\begin{align}\label{eq:estpsi}
	\widehat{\psi}_{K,n}(s)=[b^{y}_K(s)]^{-1}\widehat{{\psi}}_{K,n}^{y}(s),\quad s\in K_R.
\end{align}
\newpage
\begin{Corollary}\label{cor:uniformconvergence}
	Under the assumptions of Theorem~\ref{lemma:uniformconvergence} and if	\begin{align}\label{cond:Xboundary}
		X(s)>0,\ \forall s\in \R^d, \|s\|<R,\quad \text{almost surely,}
	\end{align}
	then $c_K^y\widehat{\psi}_{K,n}$ converges uniformly to $\psi$ on $K_{R-\varepsilon}$ for all $\varepsilon>0$, i.e.\
	\begin{align*}
		\underset{s\in K_{R-\varepsilon}}{\sup\ }|c_K^y\widehat{{\psi}}_{K,n}(s)-{\psi}(s)|\rightarrow 0,\quad \text{almost surely.}
	\end{align*}	
\end{Corollary}

\subsection{Practical issues in a non-asymptotic setting}\label{sec:estpsipract}

\begin{figure}	
	\centering
	\includegraphics[width=0.35\linewidth]{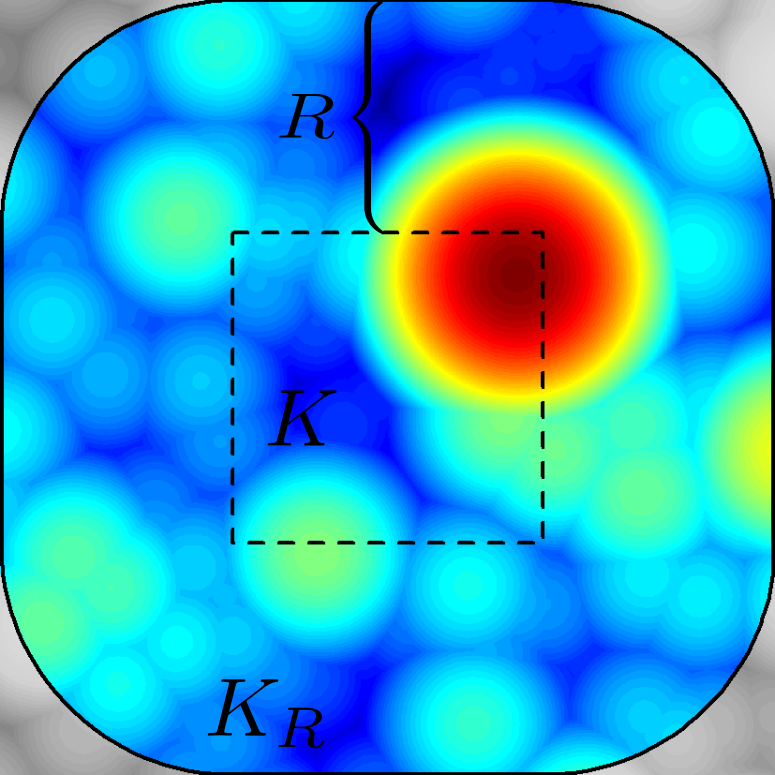}
	\caption{ Domain of observation $\mathcal{D}$ (big square) and domain of estimation $K_R$. The realization $y$ of $Y$ must be observed on $K=\mathcal{D}\ominus B_R(o)$, the process $N_K^y$ takes value on $K_R$ and its realization can be estimated by  a realization of the contributing points $N_K^Y|Y=y$.}
	\label{fig:Destcond}	
\end{figure}
One difficulty in practice is that we will mostly observe only one realization $y$ of $Y$ on some compact set $\mathcal{D}\subset\R^d$, i.e. $n=1$ in the notation of Section~\ref{sec:estpsitheo}. 
In order to obtain an estimator which is  reasonable in this setting, we slightly modify the previous approach. As before, we keep the strategy that we discuss first an estimator for $\psi^y_K=b_K^y\psi$ and then divide this estimator by $b_K^y$ to obtain the estimator for the realization $\psi$ of $\Psi$.\\
Let $K$ be the closing $K:=\mathcal{D}\ominus B_R(o)$, see Figure \ref{fig:Destcond}. The centres of the contributing storms of $y$ on $K_R$ are the points of the point process $N_K^Y|_{Y=y}$,  see Figure \ref{fig:contstorms} and \ref{fig:contstormscenterv2}. We henceforth use $N_K^Y|_{Y=y}$ as an estimate of $N_K^y=N_K^{Y^{\ast}}|_{Y^{\ast}=y}$, assuming that $N_K^Y|_{Y=y}$ approximates $N_K^y$ sufficiently well in practice - see Section~\ref{sec:discussion} for a discussion of this assumption. 		
Since $n$ equals one, we can rewrite the restriction of $c_K^y\widehat{{\psi}}_{K,n}^{y}(s)$ to any compact subset $D\subset K_R$ as
\begin{align*}
	c_K^y\widehat{{\psi}}_{K,n}^{y}(s)=h^{-d}\frac{\int_D\psi_K^y(s)~\d s}{ N_K^y(D)}\sum_{t\in N_K^y\cap D} k\left(\frac{s-t}{h}\right),\quad s\in D.
\end{align*}
The integral $\int_D\psi_K^y(s)~\d s$ is \emph{a priori} unknown and since $\E\int_D\psi_K^y(s)~\d s=\E N_K^y(D)$ we omit the whole fraction. 	
To compensate edge effects we additionally include weights $c_{D}(t)=h^{-d}\int_{D} k\left(\frac{s-t}{h}\right)\d s$ as proposed in \cite{ripley77}
and thereby obtain the following kernel estimator \citep{diggle1985kernel}		
\begin{equation}\label{est:rhohtilde}
	\widehat{{\psi}}^{y}_D(s)=h^{-d}\sum_{t\in N_K^y\cap D} c_{D}(t)^{-1} k\left(\frac{s-t}{h}\right),\quad s\in D,\quad D\subset K_R,
\end{equation}
with bandwidth $h$ and the Epanechnikov kernel	
\begin{equation*}
	k(s)=\frac{d+2}{2{c}_d}(1-\|s\|^2)\1_{B_1(o)}(s),\quad {c}_d=|B_1^d(o)|.
\end{equation*}	
The impact of the bandwidth $h$ is strong and several approaches of figuring out a reasonable bandwidth can be found in \cite{diggle1985kernel} and \cite{stoyanfraktale1992}.

\begin{Lemma}\label{lemma:unbiased}
	The estimator $\int_D\widehat{{\psi}}^{y}_D(s)~\d s$ is unbiased for 	$\int_D{\psi}^{y}(s)~\d s$, that is
	\begin{align*}
		\E\int_{D} \widehat{{\psi}}^{y}_D(s)~\d s=\E\int_{D} 	{\psi}^{y}(s)~\d s\quad \forall h\in\R_{+},\quad \forall D\subset K_R.
	\end{align*}	
\end{Lemma}
Finally, we correct $\widehat{{\psi}}^{y}_D$ as done in Corollary \ref{cor:uniformconvergence} in the previous section and use		
\begin{equation}\label{est:rhoh}
	\widehat{\psi}_D(s)=b^{y}_K(s)^{-1}\widehat{{\psi}}^{y}_D(s)
\end{equation}
to estimate $\psi$ - see Figure \ref{fig:nonparest} for illustration. 
\begin{figure}
	\centering
	\begin{minipage}[b]{5.3cm}
		\includegraphics[width=1.00\textwidth]{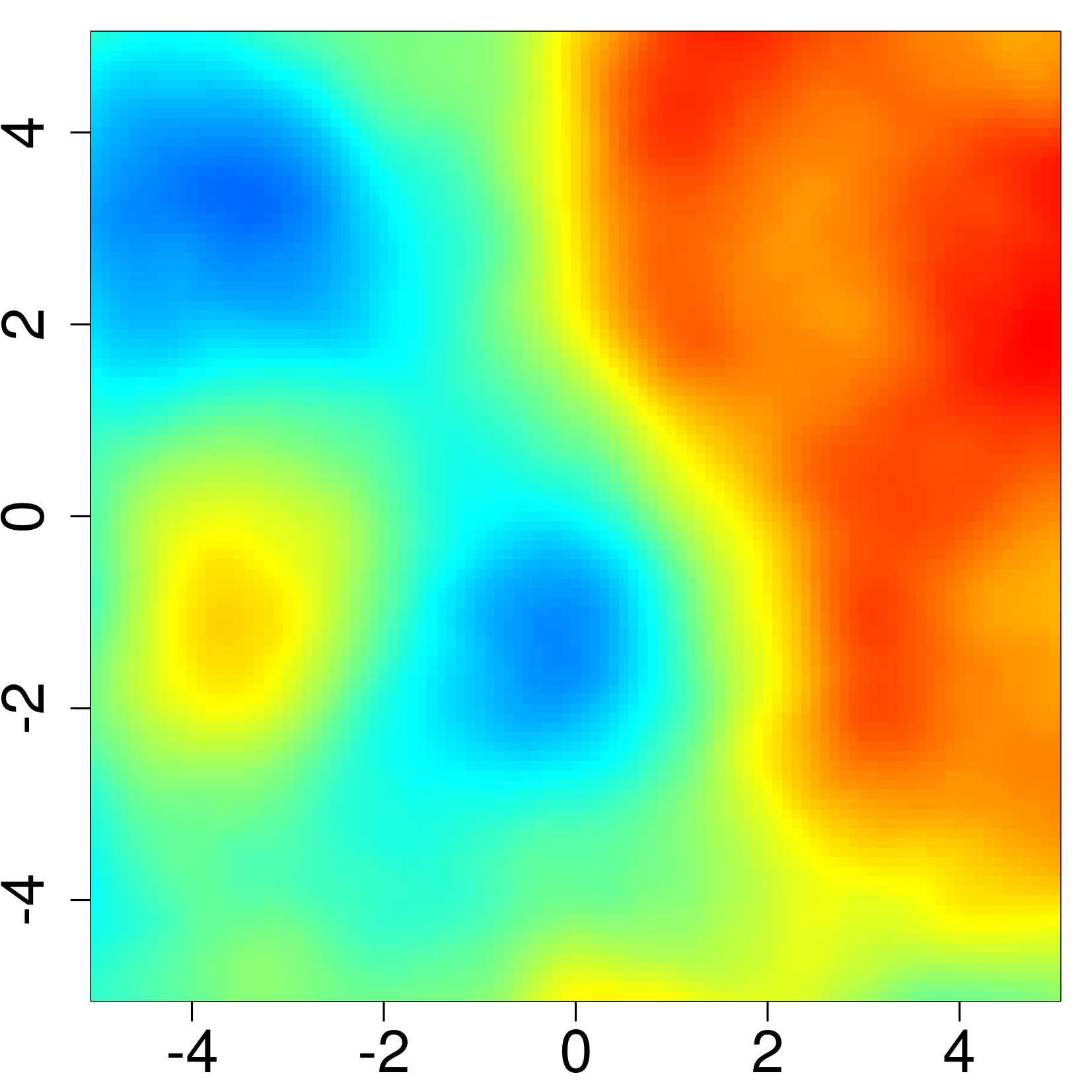}
		\includegraphics[width=1.00\textwidth]{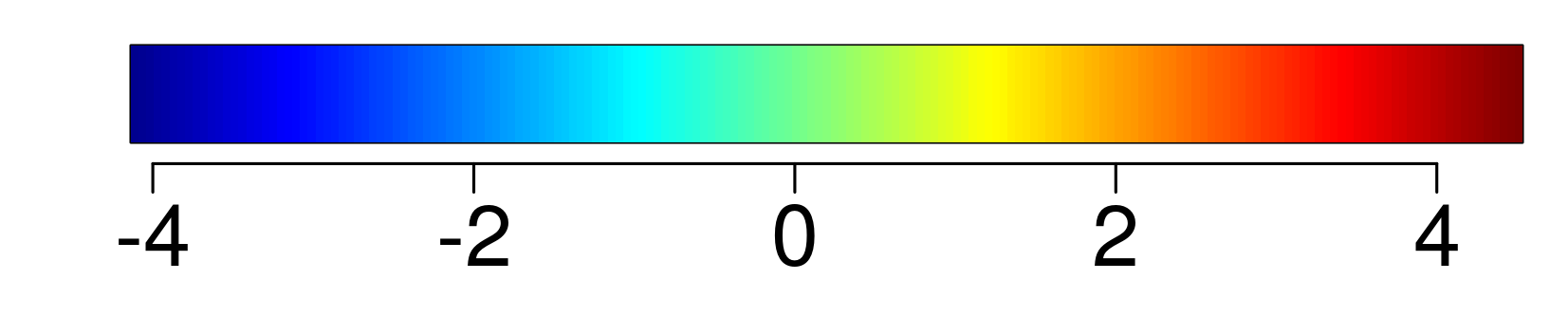}
	\end{minipage}
	\begin{minipage}[b]{5.3cm}
		\includegraphics[width=1.00\textwidth]{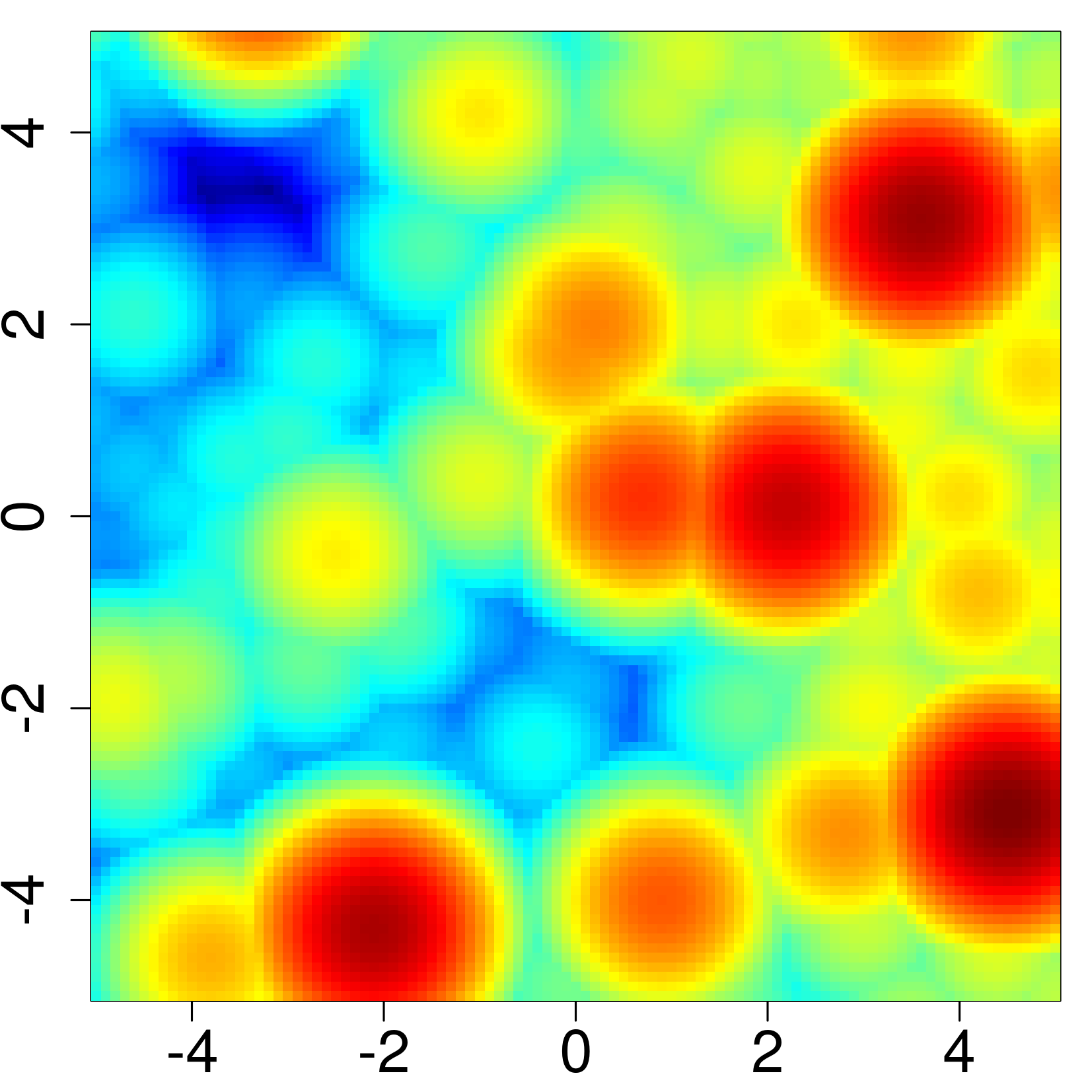}
		\includegraphics[width=1.00\textwidth]{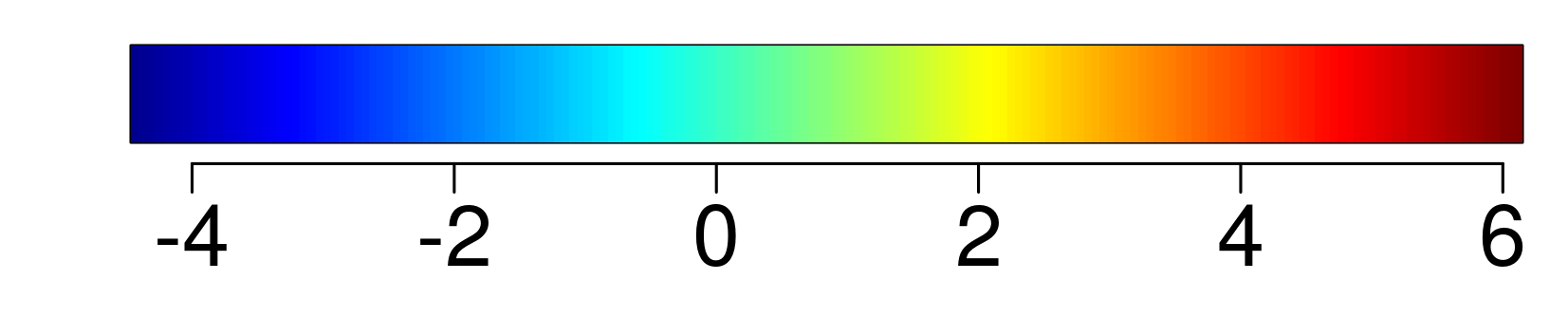}
	\end{minipage}
	\begin{minipage}[b]{5.3cm}
		\includegraphics[width=1.00\textwidth]{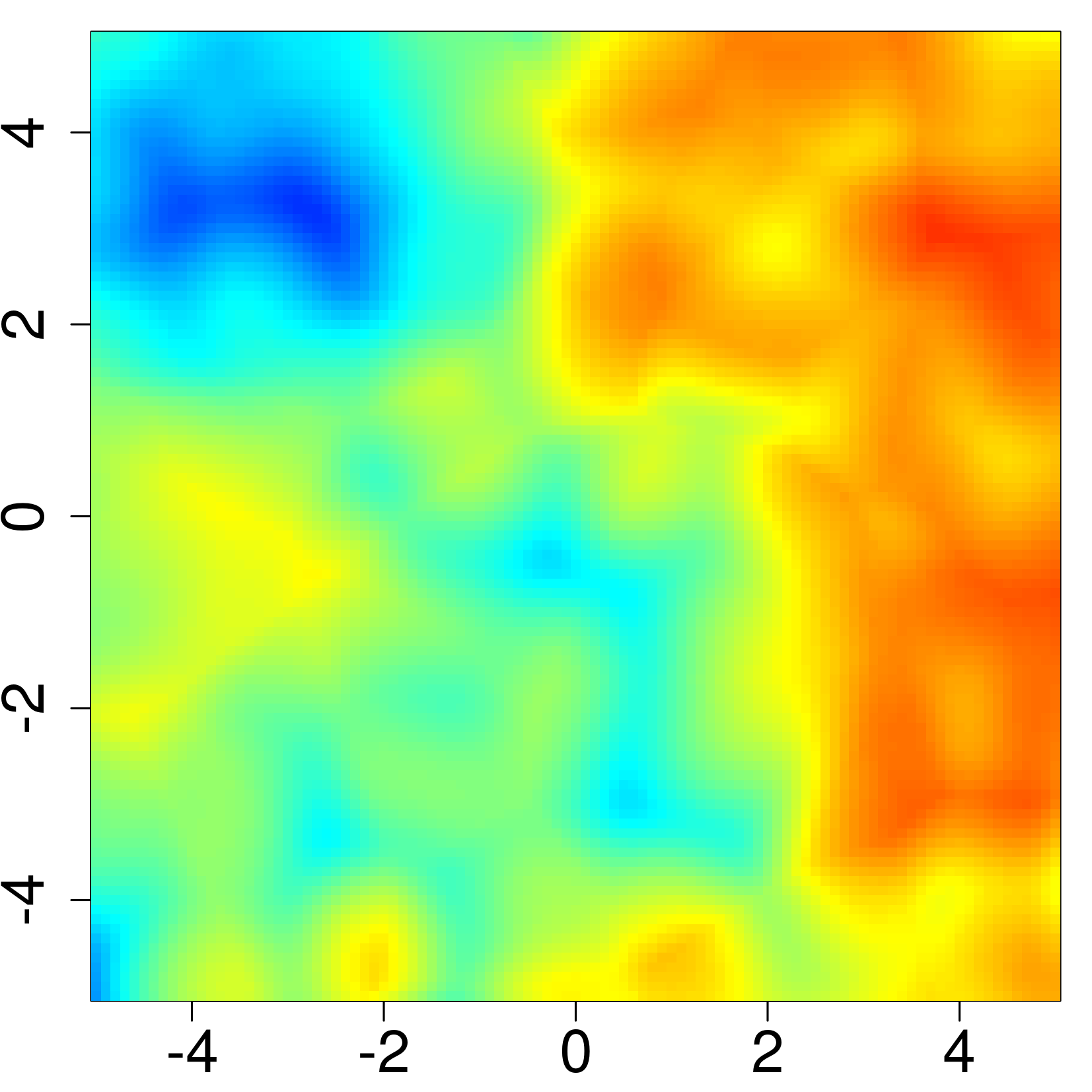}
		\includegraphics[width=1.00\textwidth]{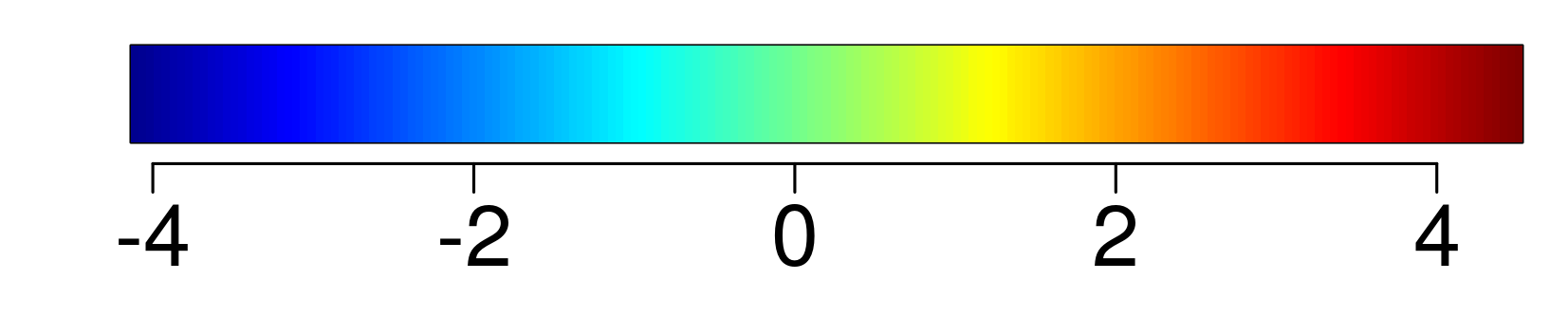}
	\end{minipage}
	\caption{True realization $\psi$ of $\Psi$ (left), associated Cox extremal process $Y$ (centre) and estimated intensity $\hat{\psi}$ (right). The intensity process $\Psi$ is log Gaussian with Matérn covariance function with parameters $\nu=2,\textrm{scale}=3,\textrm{var}=2$.}
	\label{fig:nonparest}
\end{figure}

\begin{Remark}
	The integral of the estimator $\widehat{{\psi}}^{y}_{K_R}$ is unbiased for the integral of ${\psi^y_K}$. That said, $\psi^y_K$ and $\widehat{{\psi}}^{y}_{K_R}$ are rather small near the boundary $\partial K_R$ of $K_R$. Condition \eqref{eq:Xmonotone} implies that $b^{y}_K(s)$ is also small for $s$ close to $\partial K_R$. Since ${\widehat{\psi}_{K_R}}$ is defined as ${\widehat{\psi}}_{K_R}=\widehat{{\psi}}^y_{K_R}/b^{y}_K$, the estimates are highly unstable at these areas. The severeness of this effect depends mainly on the shape function $X$ and can a priori be avoided by restricting ${\widehat{\psi}_D}$ to $D=K$ or using a smaller radius $\tilde{R}<R$ instead of the exact $R$, such that $\E\inf_{s\in B_{\tilde{R}}(o)}X(s)>\alpha$ for some sufficiently large $\alpha>0$.
\end{Remark}

\newpage
\section{Parametric estimation of the covariance function of the intensity process $\Psi$}\label{sec:estPsipara}

As described in Section~\ref{sec:model}, we model the intensity process $\Psi$ that underlies our Cox extremal process $Y$ by a log Gaussian Cox process $\Psi=\exp(W)$. Let $\sigma^2C_{\beta}$ be the covariance function of the Gaussian random field $W$ with correlation function $C_{\beta}$ and unknown parameters $\sigma^2>0$ and $\beta\in\R^p$. In the sequel we derive an estimation procedure for these unknown parameters from observations of the Cox extremal process $Y$. 

First, we modify the process of contributing storm centres such that the resulting point process behaves like the original Cox process $N_0\sim CP(\Psi(s)\d s)$. In a second step, the minimum contrast method \citep{moeller1998lgcp} is applied to these samples of $N_0$ to estimate the unknown parameters. To simplify notation, we will write $CP(\Psi)$ instead of $CP(\Psi(s)\d s)$ throughout this section and make likewise amendments for other intensity processes.\\

%

\textbf{Modification of the point process $N_K^y$.}		
\begin{figure}
	\centering
	\begin{minipage}[b]{5.3cm}
		\includegraphics[width=1.00\textwidth]{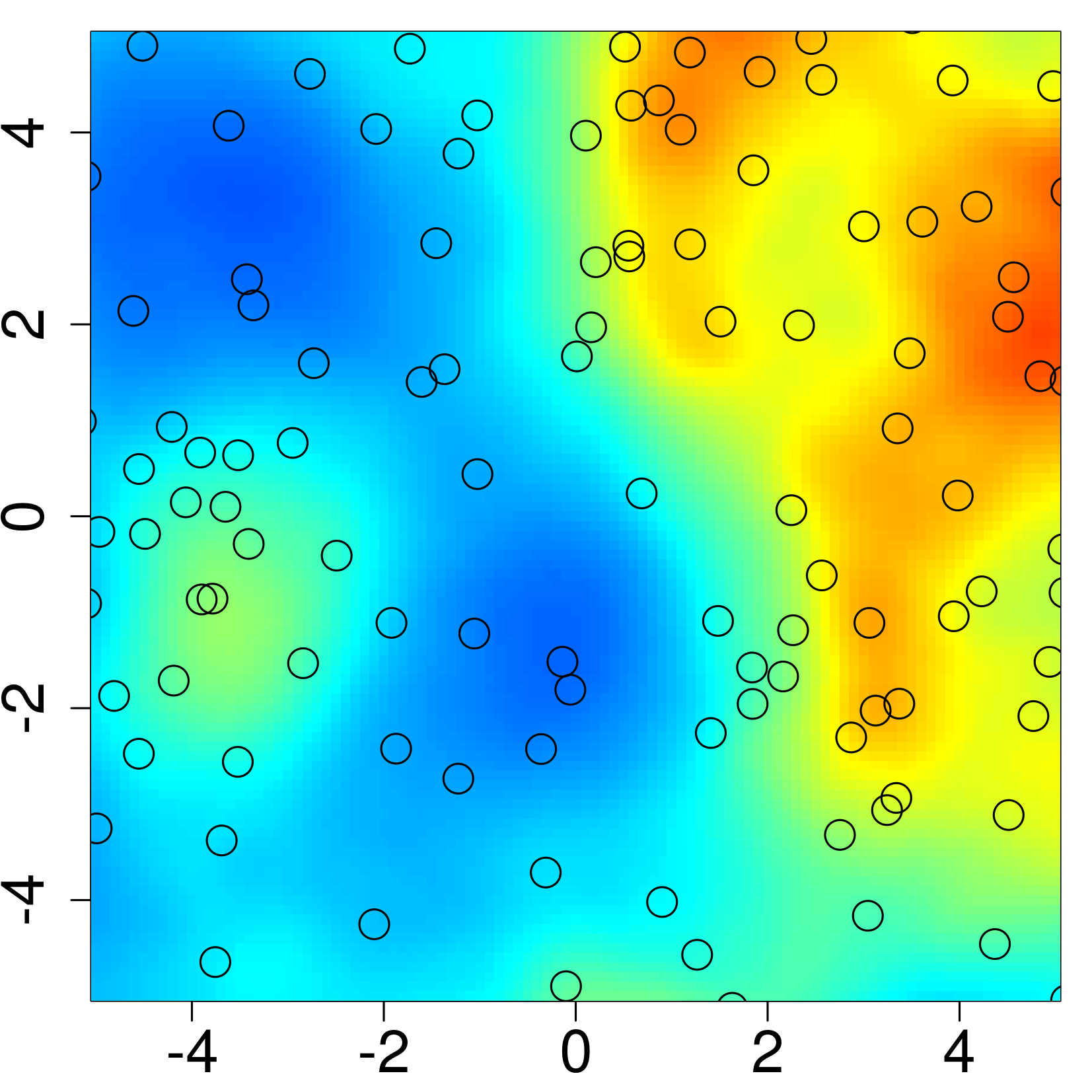}
		\includegraphics[width=1.00\textwidth]{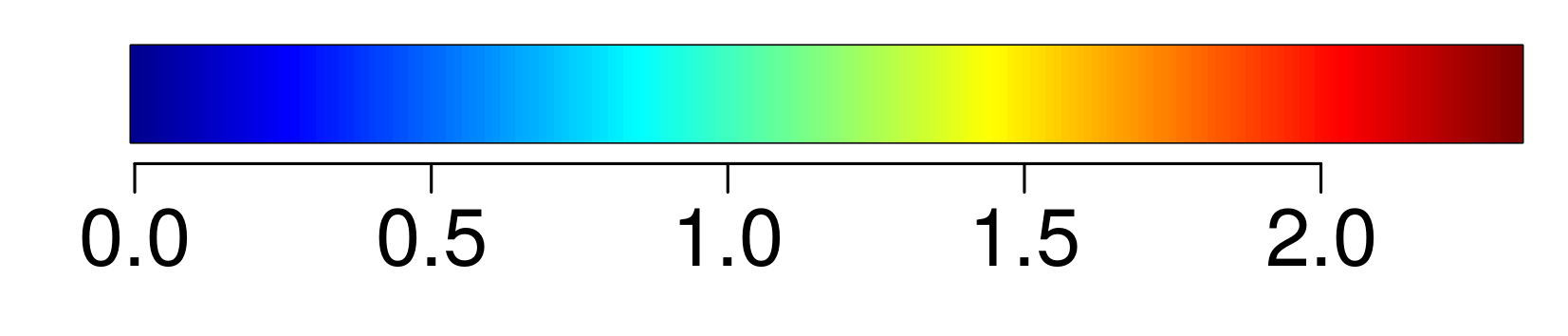}
	\end{minipage}
	\begin{minipage}[b]{5.3cm}
		\includegraphics[width=1.00\textwidth]{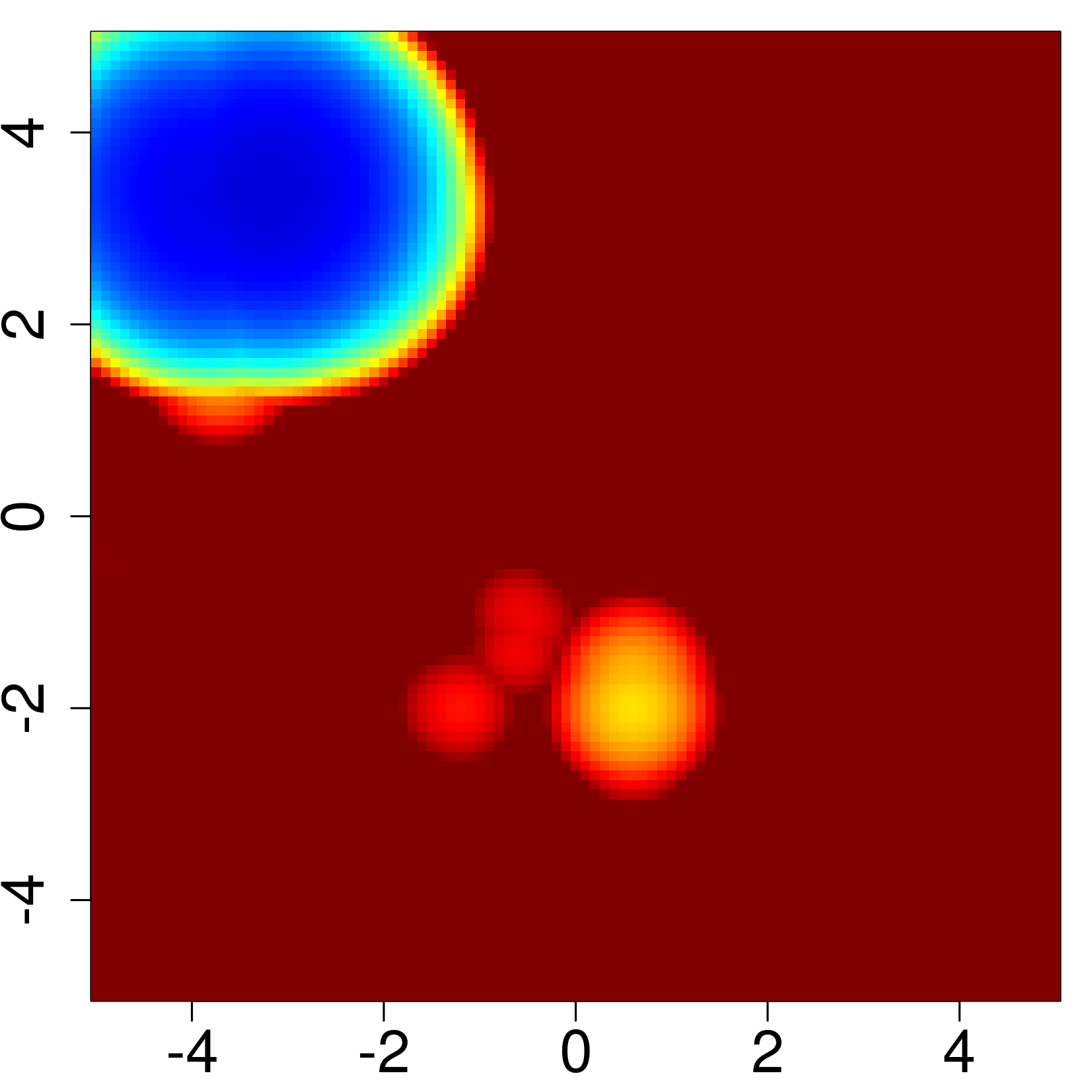}
		\includegraphics[width=1.00\textwidth]{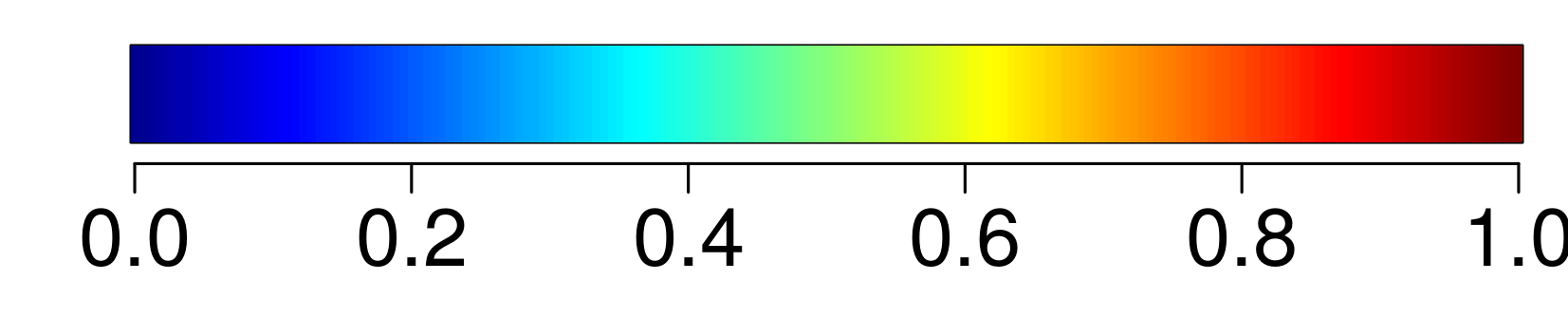}
	\end{minipage}
	\begin{minipage}[b]{5.3cm}
		\includegraphics[width=1.00\textwidth]{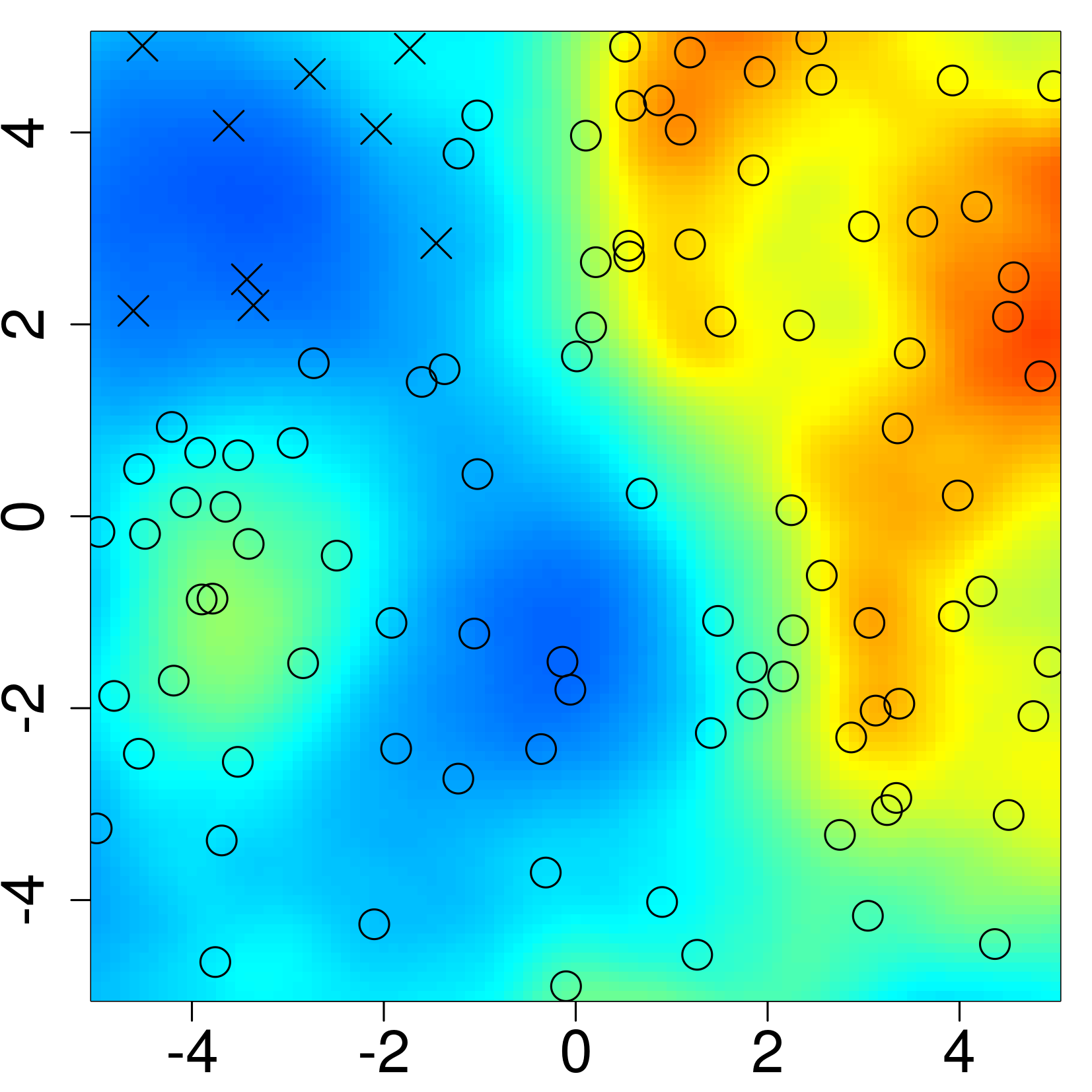}
		\includegraphics[width=1.00\textwidth]{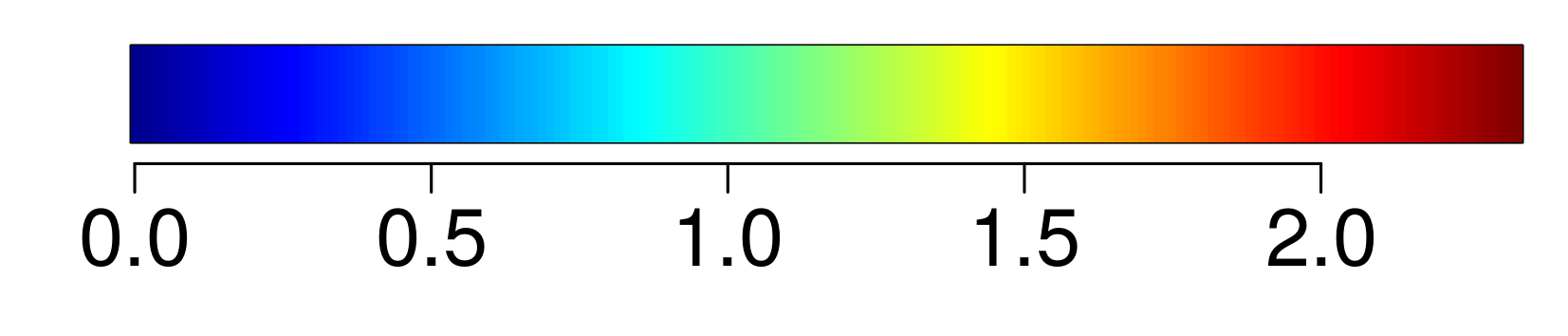}
	\end{minipage}
	\caption{Realization $\psi$ of $\Psi$ and the original sample of $N_K^y$ (left). The retaining probabilities $p$ are plotted in the middle. Thinned sample of $N_K^y$ (circles), the deleted points are marked with crosses (right).}\label{fig:enhsample1}
\end{figure}			
\begin{figure}
	\centering
	\begin{minipage}[b]{5.3cm}
		\includegraphics[width=1.00\textwidth]{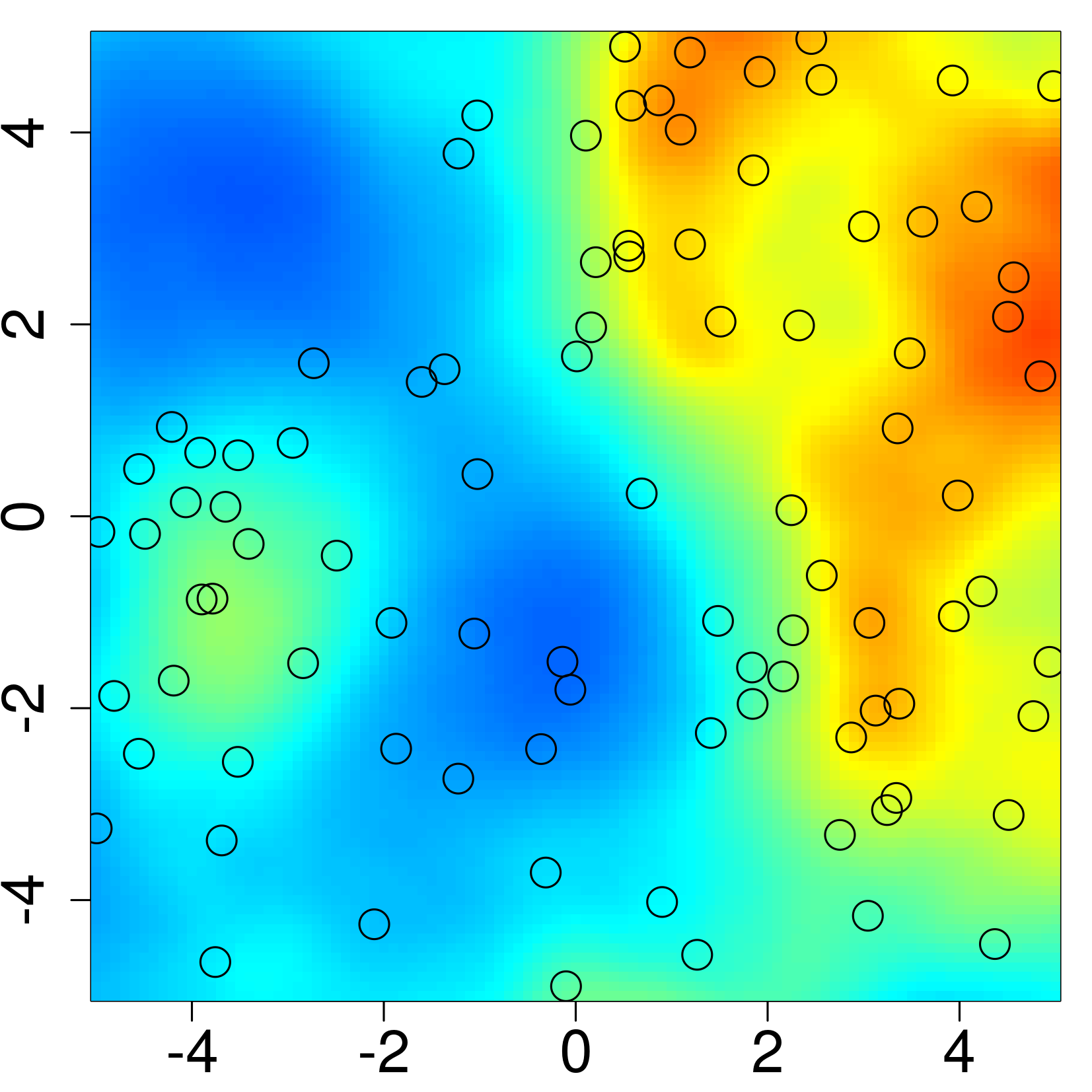}
		\includegraphics[width=1.00\textwidth]{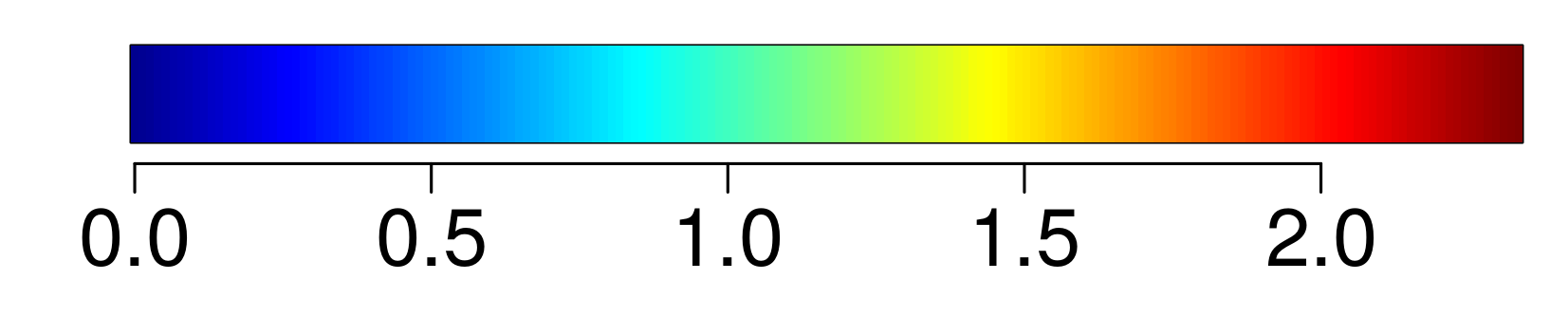}
	\end{minipage}
	\begin{minipage}[b]{5.3cm}
		\includegraphics[width=1.00\textwidth]{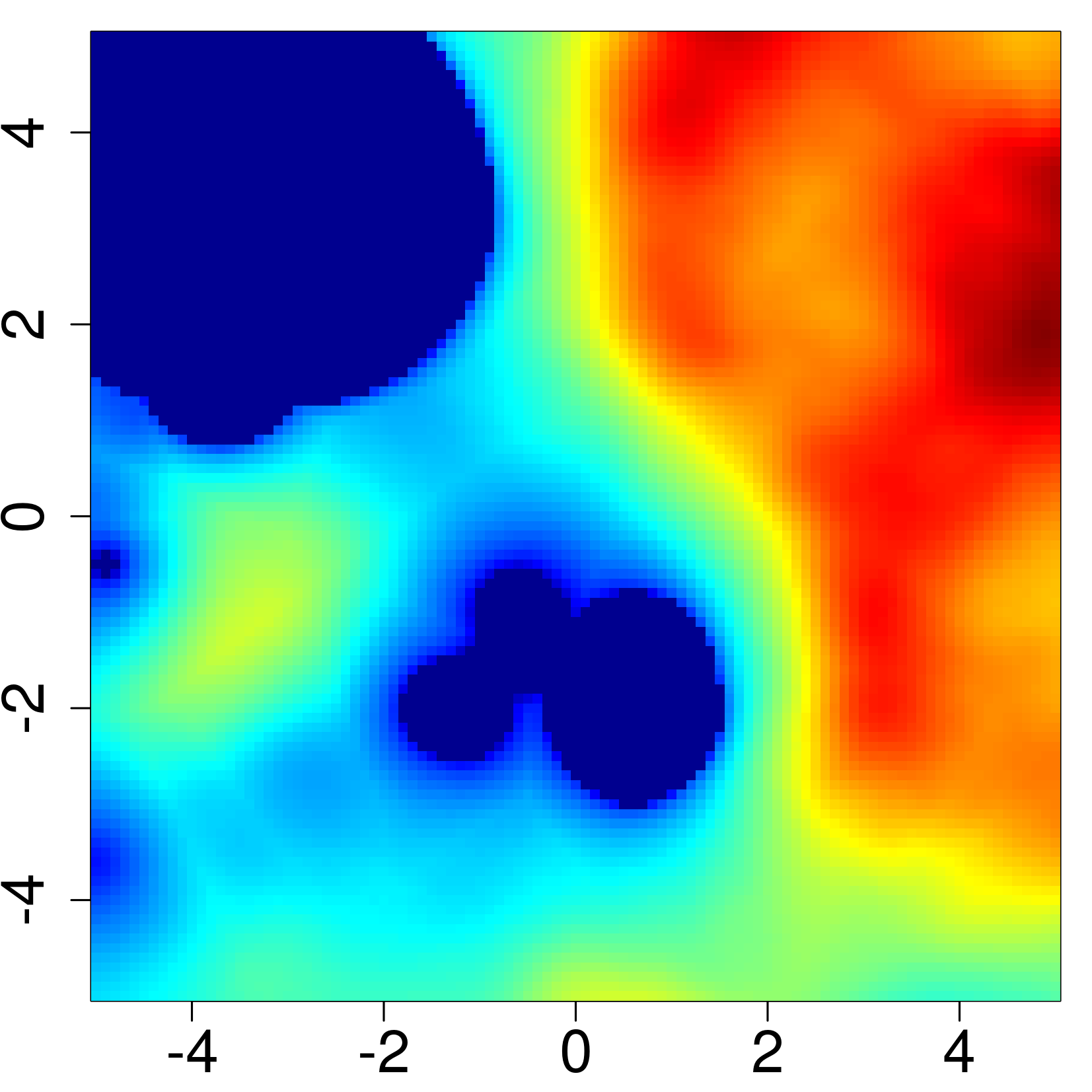}
		\includegraphics[width=1.00\textwidth]{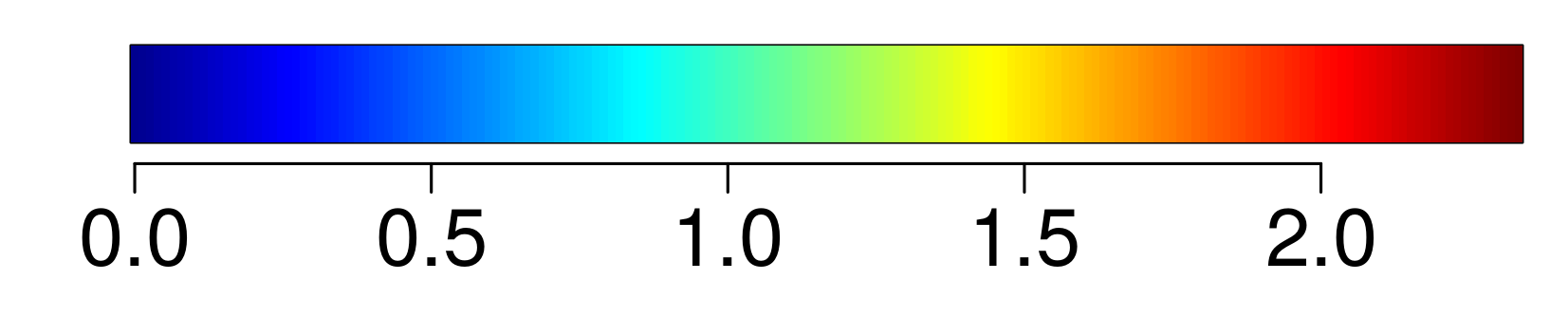}
	\end{minipage}
	\begin{minipage}[b]{5.3cm}
		\includegraphics[width=1.00\textwidth]{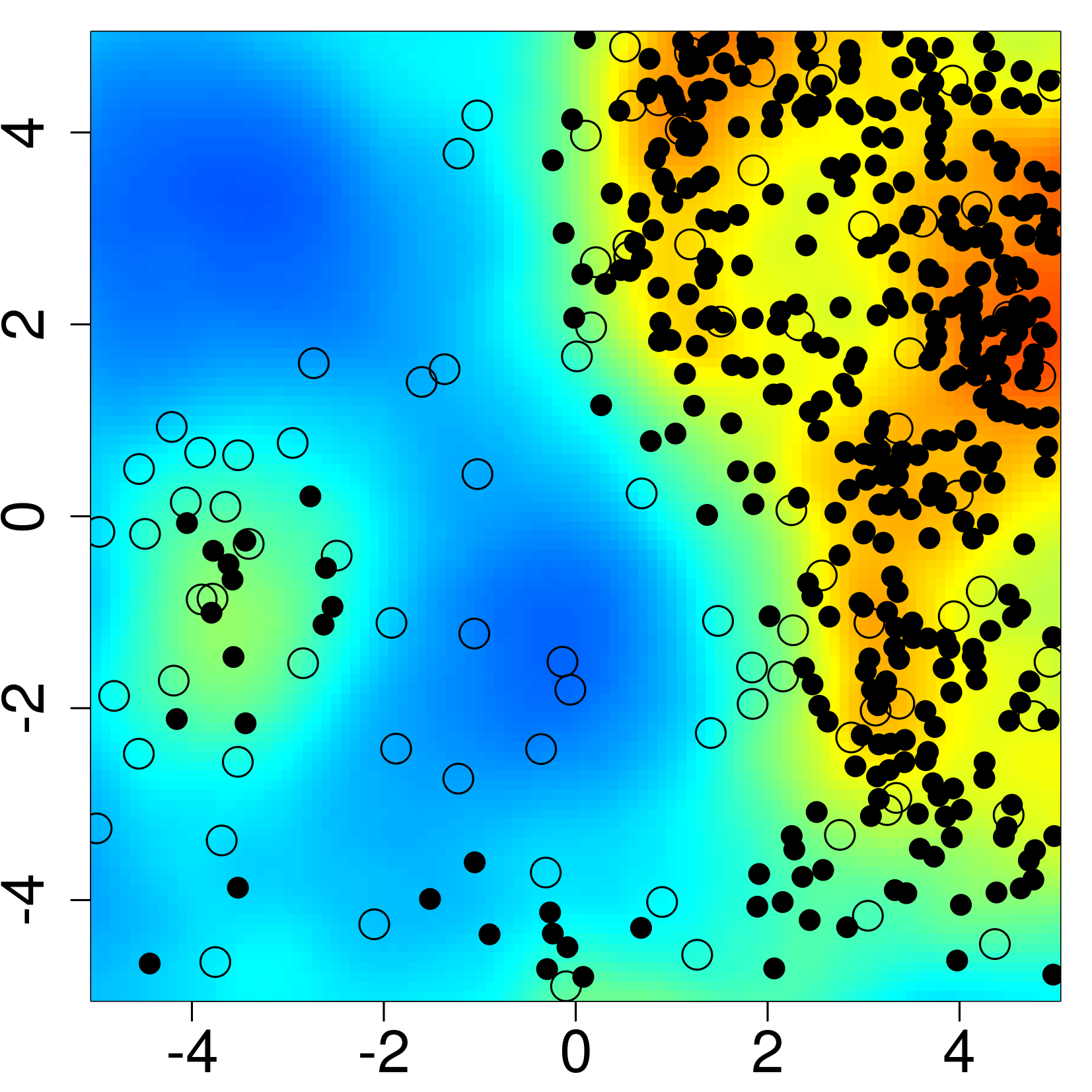}
		\includegraphics[width=1.00\textwidth]{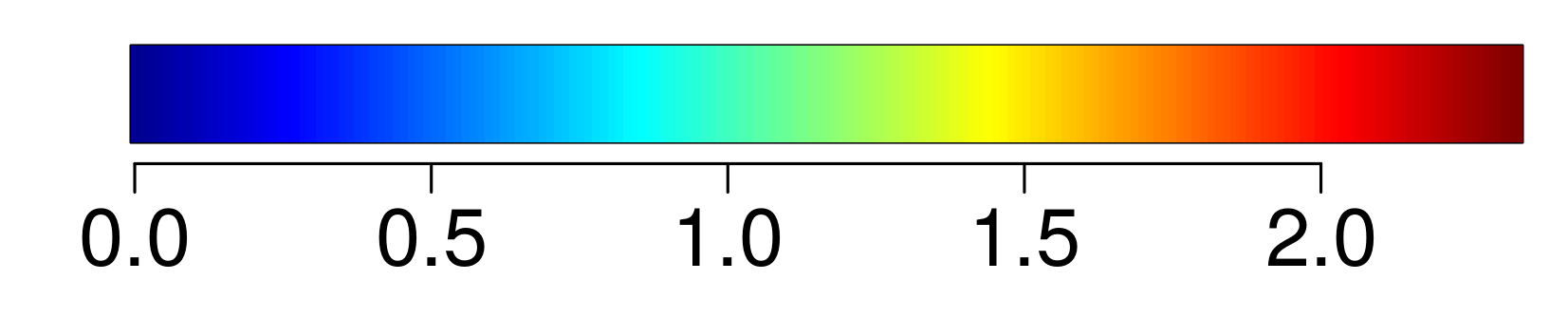}
	\end{minipage}
	\caption{Realization $\psi$ of $\Psi$ and the thinned sample $p\cdot N_K^y$ (left). Additional points are simulated with intensity function $(1-b_K^y)_{+}\Psi$ (middle). Superposition of $p\cdot N_K^y$ with the additional points (filled circles) is plotted in the right.}\label{fig:enhsample2}
\end{figure}			
As a consequence of Lemma \ref{lemma:contpointsint}, the point process $N_K^y$ that we obtain from the contributing storm centres (see Section~\ref{sec:estpsipract}) is a Cox process with intensity function $b^{y}_K\Psi$. Compared to the original point process $N_0\sim CP(\Psi)$ on which the Cox extremal process $Y$ is based, it is very likely that $N_K^y$ possesses more points in the region $\{b_K^y\geq 1\}$ and fewer points in the region $\{b_K^y<1\}$. 

To adjust for this discrepancy we delete some points of $N_K^y$ when $b_K^y>1$ and add points to $N_K^y$ when $b_K^y<1$. The first adjustment on $\{b_K^y\geq 1\}$ is done by means of thinning. If $p(\cdot)$ is a measurable function on $\R^d$ with $p(s)\in [0,1]$, then $p\cdot N_K^y$ is the point process obtained from $N_K^y$ by independent thinning according to $p(\cdot)$. That is, every point of $N_K^y$ is independently deleted with probability $1-p(\cdot)$ (see \citep{daley2008pp} Chapter 11.3 for details), see also Figure \ref{fig:enhsample1} for a plot of a sample of the original $N_K^y$, the thinning probabilities and the thinned sample $N_K^y$. We choose $p=1/b_K^y$ such that the random intensity function of the thinned process equals $\Psi$ on $\{b_K^y\geq 1\}$.
The second adjustment, adding points on $\{b_K^y< 1\}$, is achieved by simulating additional points in such way that the sum of intensity functions equals $\Psi$ on $\{b_K^Y<1\}$, see also Figure \ref{fig:enhsample2}. The following lemma summarizes and justifies this procedure.
\newpage
\begin{Lemma}\label{lemma: postprocessing}
	Let $CP(f\Psi)$ be a finite Cox process on $\R^d$ and $p=f^{-1}\cdot\1_{\{f\geq 1\}}+\1_{\{f<1\}}$. Then, $p$ is a measurable function on $\R^d$ with $p(s)\in [0,1]$ for all $s\in\R^d$ and
	\begin{align}\label{eq:CPcomp}
		\notag &p\cdot CP(f\Psi)+CP((1-f)_+\Psi)\\
		&=\underbrace{p\cdot CP(f\Psi)\vert_{\{f\geq 1\}}}_{\substack{\text{p-thinning of original } CP(f\Psi)\\ \text{on } \{f\geq 1\}}}+ \ \underbrace{CP(f\Psi)\vert_{\{f< 1\}}}_{\substack{\text{original } CP(f\Psi)\\ \text{on } \{f< 1\}}}+\ \underbrace{CP((1-f)\Psi)\vert_{\{f< 1\}}}_{\substack{\text{additional points } \\ \text{on } \{f<1\}}}\sim CP(\Psi).
	\end{align}
	That is, the left-hand side is distributed like a Cox process with intensity process $\Psi$.
\end{Lemma}
In our situation we apply the lemma to $N_K^y$ by choosing $f=b_K^y$ and restricting the resulting process to $K$. That is,
$$(p\cdot N_K^y+CP((1-b_K^y)_+\Psi))\vert_K\sim CP(\Psi)\vert_K=:\Phi_K.$$

The first two components considered in \eqref{eq:CPcomp} form the thinned point process $p\cdot N_K^y$. To add the additional points on $\{b_K^y< 1\}$ we rely on our estimate of $\psi$ from Section~\ref{sec:estpsipract}.\\


\textbf{Minimum contrast method.}		
The so-called pair correlation function \citep{stoyanfraktale1992} of a Cox process on $K\subset \R^d$ with random intensity function $\Psi=\exp(W)$ is given by
\begin{align*}
	g(s_1,s_2)=\frac{\E\left[\Psi(s_1)\Psi(s_2)\right]}{\E\Psi(s_1)\E\Psi(s_2)}, \quad s_1,s_2 \in K.
\end{align*}

A remarkable property of a log Gaussian Cox process is that its distribution is fully characterized by its first and second order product density. We refer to Theorem 1 in \cite{moeller1998lgcp}, which also covers the following lemma.

\newpage

\begin{Lemma}[Stationarity and second order properties]\label{lemma:CPsecondorder}
	A log Gaussian Cox process is stationary if and only if the corresponding Gaussian random field is stationary. Then, its pair correlation function equals
	\begin{align*}g(s_1-s_2)=\exp(\sigma^2 C(s_1-s_2)),\end{align*}
	where $\sigma^2C(\cdot)$ is the covariance function of the associated Gaussian random field.
\end{Lemma}

Hence, a log Gaussian Cox process enables a one-to-one mapping between its pair correlation function and the covariance function of the associated Gaussian random field. Therefore, the spatial random effects influencing the random intensity function of the Cox process can be studied by properties of the Cox process itself.
The minimum contrast method \citep{diggle1984mcm, moeller1998lgcp} exploits this fact.

\begin{Proposition}[Minimum contrast method, \citep{moeller1998lgcp}]\label{prop:mincontrast}$\phantom{a}$
	Suppose that $T_{\sigma^2,\beta}(h)=\sigma^2C_{\beta}(h)$ is the covariance function of a Gaussian random field $W$. Let $g$ be the pair correlation function of the log Gaussian Cox process associated to $W$. If $\widehat{g}$ is an estimator for $g$ and $\widehat{T}(h)=\log \widehat{g}(h)$, then the distance 	
	\begin{align}\label{Minimum Contrast}
		d(T_{\sigma^2, \beta},\hat{T})=\int_{\varepsilon}^{r_0} \bigg(T_{\sigma^2, \beta}(r)^{\alpha}-\hat{T}(r)^{\alpha}\bigg)^2\d r,
	\end{align}
	with tuning parameters $0 \leq \varepsilon < r_0$ and $\alpha>0$, is minimized by the minimum contrast estimators
	\begin{align}\label{eq:mincontrast}
		\hat{\beta}=\underset{\beta}{\argmax\ }\frac{A(\beta)^2}{B(\beta)}, \quad \hat{\sigma}^2=\left(\frac{A(\hat{\beta})}{B(\hat{\beta})}\right)^{1/\alpha},\end{align}	
	with 	
	\begin{align*}A(\beta)=\int_{\varepsilon}^{r_0}{\big[\log\big(\hat{g}(r)\big)C_{\beta}(r)\big]^{\alpha}}\d r, \quad B(\beta)=\int_{\varepsilon}^{r_0}{C_{\beta}(r)^{2\alpha}}\d r.\end{align*}
	
\end{Proposition}
The minimum contrast method minimizes the distance of the pair correlation function $g$ and its estimator $\widehat{g}$. Thus, the task of estimating the covariance parameters of $W$ is transformed to a non-parametric estimation of $g$. \\

\textbf{Combined procedure for estimation of $\beta$ and $\sigma^2$.}
We consider the same setting as in Section \ref{sec:estpsipract}, that is, an observation $y$ of $Y$ is given on a compact set $\mathcal{D}$. We set $K=\mathcal{D}\ominus B_R(o)$ and approximate $N_K^y\approx N_K^{Y}|_{Y=y}$.
Now Lemma~\ref{lemma: postprocessing} justifies to approximate a realization of $\Phi_K$ by a realization of $$\widehat{\Phi}_K=p\cdot N_K^y+PP((1-b_K^y)_+\widehat{\psi}_{K_R})\vert_K,$$
by simulating additional points from the point process $PP(\1_{K_0}(1-b)\widehat{\psi}_{K_R})\vert_K$, where $\widehat{\psi}_{K_R}$ is the estimator described in Section \ref{sec:estpsipract}. Next, we estimate the pair correlation function $g$ of $\Phi_K$ by a non-parametric kernel estimator based on the realization $\widehat{\phi}_K$ of $\widehat{\Phi}_K$. Finally, the minimum contrast method can be applied to $\widehat{g}$ to obtain estimates of the parameters $\sigma^2$ and $\beta$ of the log Gaussian Cox process.


\begin{Remark}
	Besides using $\widehat{\phi}_K$ to estimate $g$, it is also possible to simulate $\widetilde{\phi}_K\sim PP(\widehat{\psi}_{K_R}(s)~\d s)$ on the whole set $K$ and build estimators for $g$ from samples of $\widetilde{\phi}_K$. However, this leads to a higher bias since the intensity of $\widehat{\phi}_K$ is exactly equal to $\psi$ on $\{b_K^y\geq 1\}$ if $N_K^y$ is known. Additionally, computing the thinning of $N_K^y$ on $\{b_K^y\geq 1\}$ is much faster than simulating $PP(\widehat{\psi}_{K_R}(s)~\d s)$ on $\{b_K^y\geq 1\}$.
\end{Remark}

\begin{figure}[bth]
	\centering
	\includegraphics[width=0.35\linewidth]{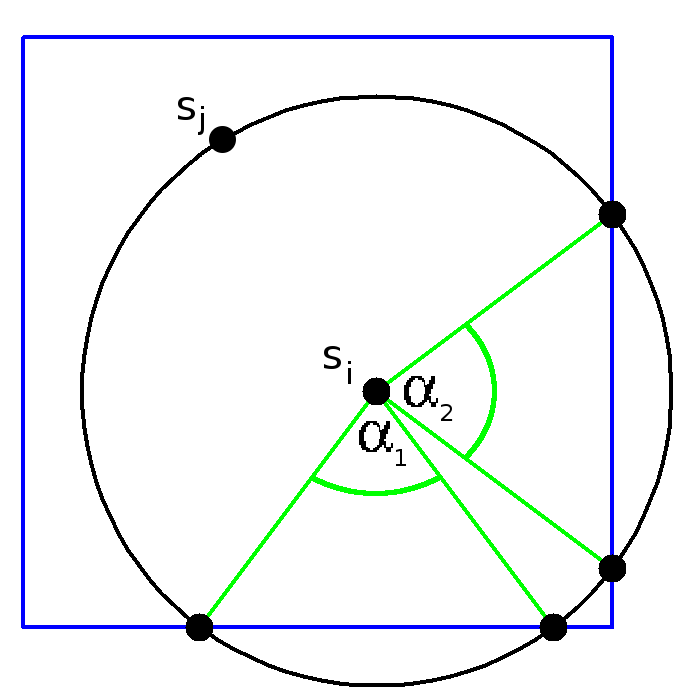}
	\caption{The edge correction $b_{ij}$ is the ratio of the whole circumference $2\pi$ and the circle arcs $\gamma_{ij}=2\pi-\alpha_1-\alpha_2$ within the square $K$.}
	\label{fig:kernelweights}
\end{figure}

\textbf{Practical aspects of implementation.} We propose to use a non-parametric kernel estimator as discussed by \cite{stoyanfraktale1992} (Part III, Chapter 5.4.2). Consider $\widehat{\phi}_K=\sum_{i=1}^n \delta_{s_i}$, then we estimate the pair correlation function $g$ by		
\begin{equation*}
	\widehat{g}(r)=\frac{|K|}{2\pi n^2 r}\sum_{\substack{i,j=1\\i\neq j}}^n k_h(r-\|s_i-s_j\|)b_{ij},
\end{equation*}		
with the Epanechnikov kernel $k_h(r)=0.75h^{-1}(1-r^2/h^2)\1_{|r|<h}$,
and kernel weights $b_{ij} \geq 0$ for edge correction (see \cite{ripley77}). These are defined as $b_{ij}={2\pi}/{\gamma_{ij}}$ which is the ratio of the whole circumference of $B_{\|s_i-s_j\|}(s)$ to the circumference within $K$, i.e $\gamma_{ij}$ is the sum of all angles, for which the associated non-overlapping circle arcs are within $K$, see Figure~\ref{fig:kernelweights}.

\begin{Remark}
	The estimates of $\sigma^2$ and $\beta$ obtained from $\widehat{g}$ by the minimum contrast method, have a very high variance. Therefore, this procedure is only recommended if we observe several i.i.d.\ realizations $y_1,\dots, y_n$ of $Y$ and the associated $N_K^{y_1},\dots, N_K^{y_n}$. The final estimates of $\sigma^2$ and $\beta$ may then be defined as the mean or median of the estimates obtained from $\widehat{g}_1,\dots, \widehat{g}_n$.
\end{Remark}

\begin{figure}[tbh]		
	\centering
	\begin{minipage}[b]{7.7cm}
		\includegraphics[width=1\textwidth]{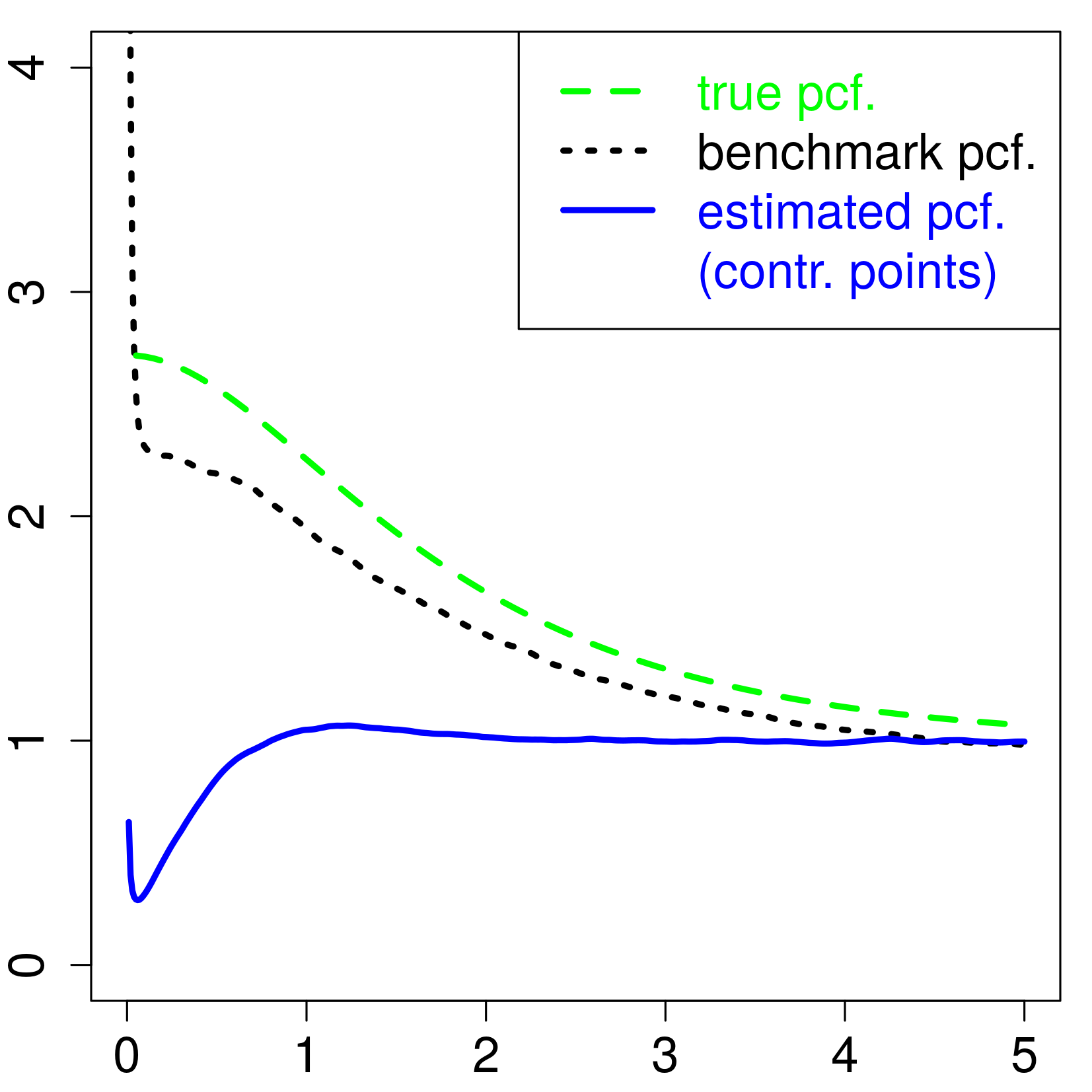}
	\end{minipage}
	\hspace{5mm}
	\begin{minipage}[b]{7.7cm}
		\includegraphics[width=1\textwidth]{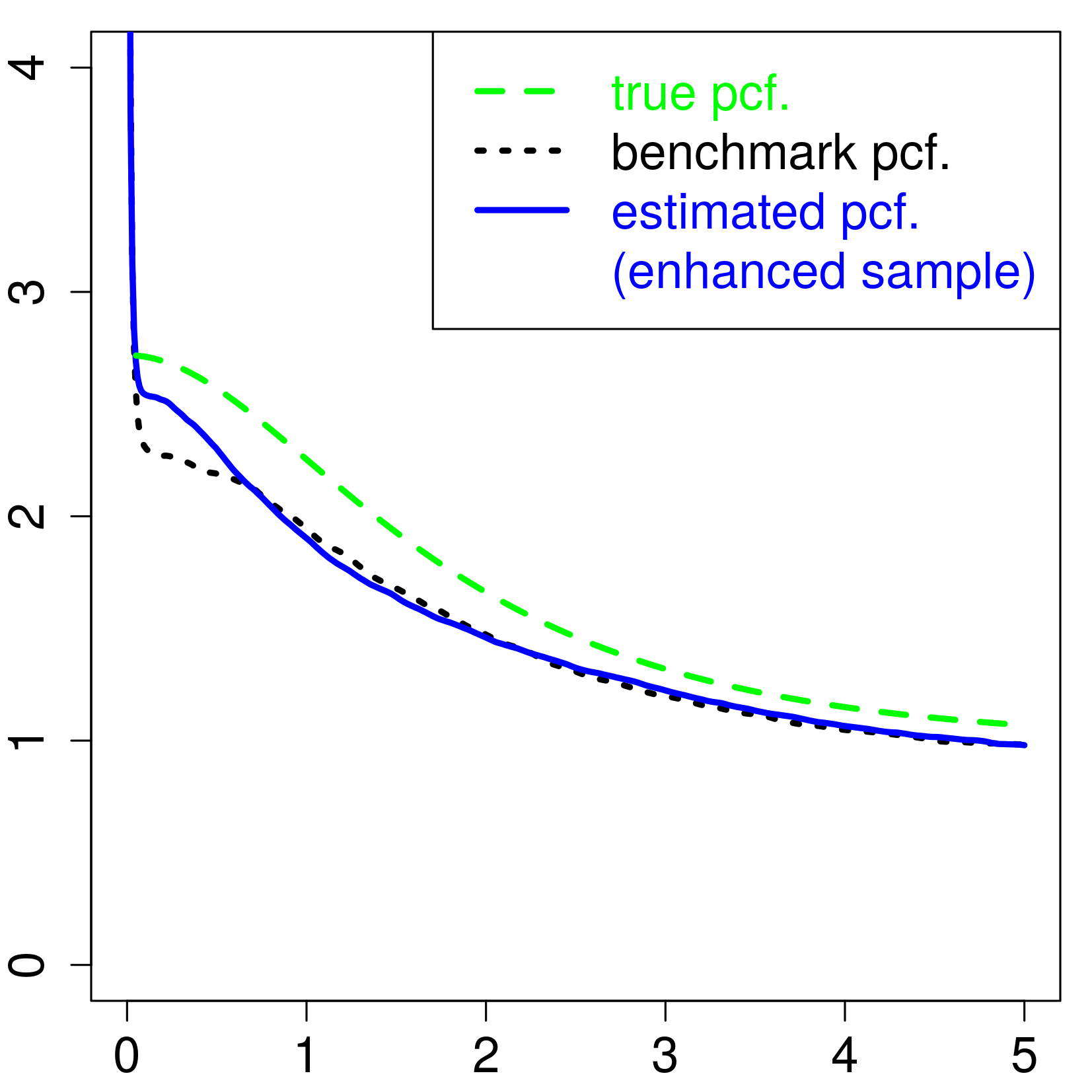}
	\end{minipage}	
	\caption{Estimated pair correlation functions via a sample of $N_K^{Y}|_{Y=y}$ (left) and by points of the modified process $\widehat{\Phi}_K$ (right). They are pointwise averages of $n=50$ experiments.}	\label{fig:paircorr}			
\end{figure}

Plots of the estimated pair correlation functions via a sample of $N_K^{Y}|_{Y=y}$ and by points of $\widehat{\Phi}_K$ are compared in Figure~\ref{fig:paircorr}. They are also compared to the true pair correlation function and the natural benchmark which is obtained from a direct sample of $\Phi_K$ instead of $\widehat{\Phi}_K $. Numerical experiments such as reported in Figure~\ref{fig:paircorr} and Section~\ref{sec:simulationstudy} support that our proposed modification works surprisingly well.

\section{Simulation study}\label{sec:simulationstudy}

We survey the performance of our proposed non-parametric estimator $\widehat{\psi}_D$ \eqref{est:rhoh} of the realization $\psi$ of $\Psi$ and that of the estimators $\widehat{\beta}$ and $\widehat{\sigma^2}$ \eqref{eq:mincontrast} of the parameters of the covariance function of $\Psi=\exp(W)$ in a simulation study.\\

\textbf{Setting.} In our numerical experiments we choose the covariance of the underlying Gaussian process $W$ to be the Whittle-Mat\'ern model \eqref{eq:matern} with known smoothness parameter $\nu~\in~\{0.5,1,2,\infty\}$ and unknown variance $\sigma^2$ and scale $\beta$. The scale $\beta$ will control the size of clusters in our point processes and the variance $\sigma^2$ directs the variability of the number of points within the local clusters. The performances of the associated estimators are compared for different choices $\sigma^2\in\{1,2\}$ and $\beta\in\{1,2\}$. As shape mechanism we consider the fixed storm process $X=\varphi$ where $\varphi$ is the density of the standard normal distribution. We simulate $n=1000$ realizations $y_1,\dots,y_n$ of the corresponding Cox extremal process $Y$ on an equidistant grid with $101^2$ grid points in $[-5,5]^2$. 

Henceforth, we simplify the notation from the previous section by writing $\widehat{\psi}$ instead of $\widehat{\psi}_{D}$ for the estimated intensity function.
A natural benchmark of our estimation procedures from Sections \ref{sec:estpsinonpara} and \ref{sec:estPsipara} are such estimators which are derived from direct samples of a Cox process $N_0~\sim~CP(\Psi(s)~\d s)$ with spatial intensity process $\Psi$.
We denote the benchmark kernel estimator by
\begin{equation}\label{eq:benchmarkest}
	\widehat{{\psi}_0}(s)=h^{-d}\sum_{t\in N_0\cap D} c_{D,h}(t)^{-1} k\left(\frac{s-t}{h}\right),\quad s\in D.
\end{equation}
Accordingly, let $\widehat{\sigma^2_0}$ and $\widehat{\beta}_0$ be the minimum contrast estimators obtained from direct samples of $N_0$.

\vspace{2mm}

\newpage
\textbf{Measures for evaluation.} To assess the performance of our non-parametric estimates we use the following \textit{mean relative variance} 
\begin{align}\label{eq:MRV}
	\widehat{\MRV}(\widehat{\psi},\psi):=n^{-1}|D|^{-1}\sum_{i=1}^n\sum_{s\in D}(c_i\widehat{\psi_i}(s)/\psi_i(s)-1)^2,
\end{align}
with $c_i^{-1}=|D|^{-1}\sum_{s\in D}\frac{\widehat{\psi_i}(s)}{\psi_i(s)}$ for $i=1,\dots,n$.
Up to a scaling constant, $\widehat{\MRV}(\widehat{\psi},\psi)$ is an empirical version of $\MISE(c\cdot\widehat{\psi}/\psi,1)$ with $\MISE(\widehat{\phi},\phi):=\E\int(\widehat{\phi}(s)-\phi(s))^2~\d s.$
We compare the $\MRV$ of our estimated intensity $\hat{\psi}$ with that of the benchmark $\widehat{\psi}_0$. The corresponding relative $\MRV$ of $\widehat{\psi}_0$ and $\widehat{\psi}$ is defined as the ratio $\widehat{\MRV}(\widehat{\psi})/\widehat{\MRV}(\widehat{\psi_0})$.\\
The goodness of fit of the parametric estimates is measured in terms of the empirical mean squared error
$\widehat{\MSE}(\widehat{\theta}):=~n^{-1}\sum_{i=1}^n(\widehat{\theta}_i-~\theta)^2$.
Again the $\MSE$ of $\widehat{\sigma^2}$ and $\widehat{\beta}$ are compared with those of the benchmark estimators $\widehat{\sigma^2_0}$ and $\widehat{\beta}_0$, respectively.\\

\textbf{Results.} 
The results of our simulation study are reported in the tables of Figures~\ref{fig:simunonpar} and \ref{fig:simupar} in Appendix~\ref{sec:app_results}.
The best performance we can hope for is to be as good as the benchmark estimators that are applied to samples of the original point process $N_0 \sim CP(\Psi(s)ds)$.
Hence, in the case of our non-parametric estimation of the realisations of the intensity processes we can expect the ratios $\widehat{\MRV}(\widehat{\psi})/\widehat{\MRV}(\widehat{\psi_0})$ to be always greater or equal to $1$ and at best even close to $1$. Indeed, this is confirmed by the simulation study as can be seen from Figure~\ref{fig:simunonpar}.
All ratios (except one) lie slightly above 1. The exceptional case occurs when the standard error of $\widehat{\MRV}(\widehat{\psi_0})$ is relatively high, where we even outperform the benchmark. This is quite remarkable given that we infer the intensity under an independence assumption that is not necessarily satisfied (cf.\ Section~\ref{sec:discussion} for a discussion) and secondly, we correct it by a data driven quotient as in \eqref{eq:by}. 

Likewise we observe that the standard errors for the MRV are close to the benchmark when $\beta=1$  and much smaller -- sometimes even half the size -- in the case $\beta=2$. This indicates that our estimation procedure for $\psi$ is relatively stable compared to the benchmark. In general, both estimators perform better for the larger value of the scale parameter $\beta$, that is for larger cluster sizes in the point processes, whereas a higher variance $\sigma^2$ naturally leads to a worse performance. The influence of the smoothness parameter is not entirely clear. Looking at the values  $\nu\in\{0.5,1,2\}$ one might conclude that the estimation improves for a smoother intensity. But in case of the smoothest field ($\nu=\infty$) the MRVs get larger again. 
However, what is more important is that both procedures, the one that we proposed for inference on $\psi$ for Cox extremal processes and the benchmark $\widehat \psi_0$, behave coherently as the parameters vary across different smoothness classes, cluster sizes and variability of number of points within local clusters. 

The non-parametric estimates are further used to obtain the parametric estimates of $\sigma^2$ and $\beta$. Here, estimation of the pair correlation function is very sensitive to the choice of the scale. Our maximal scale $\beta=2$ is large in relation to the size of the observation window $[-5,5]^2$ which causes a bias in the estimation of all pair correlation functions. Therefore, all parametric estimates -- the benchmarks $\widehat{\sigma^2}_0,\widehat{\beta}_0$ as well as our estimates $\widehat{\sigma^2},\widehat{\beta}$ -- are also biased when $\beta=2$. The estimation of $\sigma^2$ is volatile if $\sigma^2=2$, this applies in particular to our $\widehat{\sigma^2}$ which fails when both $\beta=1$ and $\sigma^2=2$. Still, in all other cases the $\MSE$ of our multi-stage estimators is close to that of the benchmark. There are even some cases when we outperform the benchmark, which is not surprising as the standard errors are very high in general.

\section{Discussion}\label{sec:discussion}

In this article we present a new class of conditionally max-stable random fields based on Cox processes, which we therefore also call Cox extremal processes. We prove in Theorem \ref{thm:domainofattraction} that these processes are in the MDA of familiar max-stable models. Hence, they have the potential to model spatial extremes on a smaller time scale.

An objective of practical importance is to identify the random effects influencing the underlying Cox process from the centres of the contributing storms. In order to make inference feasible, we impose an additional independence assumption on our observed data (see Lemma \ref{lemma:contpointsint}) that allows to derive a uniformly consistent non-parametric kernel estimator \eqref{eq:estpsi} for the realization $\psi$ of the intensity process $\Psi$ (Corollary \ref{cor:uniformconvergence}). 
Imposing such an independence assumption can be seen in a similar manner to the composite likelihood method that ignores dependence among higher order tupels.
We believe that our condition is sufficiently well satisfied in most situations, since only a small number of large storms from the Cox process $N$ approximate the Cox extremal process $Y$ already quite well. Practical adjustments of the estimator \eqref{eq:estpsi}, in particular to a non-asymptotic setting, are presented in Section~\ref{sec:estpsipract}. 

For parametric estimation the non-parametric estimator \eqref{est:rhoh} can be used to correct the observed point processes $N_K^y$ of contributing storm centres in order to obtain samples from $CP(\Psi)$ (Lemma~\ref{lemma: postprocessing}). If $\Psi$ is log Gaussian, the minimum contrast method can be applied subsequently to obtain estimates for the parameters of the covariance function of $\log \Psi$. 

The performance of our proposed estimation procedures is addressed in a simulation study (Section~\ref{sec:simulationstudy}).
Here, the best we can hope for is that our estimators can compete with the benchmark estimators applied to the original point process $CP(\Psi(s)ds)$. Indeed, our non-parametric procedure is usually relatively close to the benchmark which is quite remarkable in view of the necessary adjustments we have to make.
Also, looking at different kinds of smoothness, cluster sizes and variances we find evidence for the stability of our proposed estimation procedure when compared to the benchmark.
Both (our procedure and the benchmark) behave coherently across different choices of these properties. Similar behaviour can be observed for the parametric estimates, even though they are more volatile and the estimation of the pair correlation function is generally very sensitive to the choice of scale.

Within our simulation study and all other illustrations, we consider deterministic storm processes $X=\varphi$ where $\varphi$ is the density of the standard normal distribution. This restriction is only done to reduce the computing time. 
Indeed, all estimators presented in Section~\ref{sec:estpsinonpara} and \ref{sec:estPsipara} are valid for much more general $X$ and simulations showed that the specification of $X$ only slightly influences the inference on $\Psi$ as long as enough centres of contributing storms can be identified. For instance, if we impose the monotonicity assumption \eqref{eq:Xmonotone} on the storm process $X$, the majority can be recovered as local maxima of the realization $y$ of $Y$. Computational methods for identification of such points are left for further research.

\vspace{5mm}

{\small
	\textbf{Acknowledgments.} The research of MD was partly supported by the DFG through 'RTG 1953 - Statistical Modeling of Complex Systems and Processes' and Volkswagen Stiftung within the project 'Mesoscale Weather Extremes - Theory, Spatial Modeling and Prediction'. 
	The authors are grateful to A.~Baddeley for useful comments on the estimation of pair correlation functions and thank M.~Oesting for a discussion of simulation algorithms for max-stable processes.
}
\newpage
\appendix

\section{Appendix}\label{sec:appendix}
\enlargethispage{100mm}
\subsection{Simulation results}\label{sec:app_results}

\vspace*{-4mm}
\begin{figure}[H]
	\centering
	\includegraphics[width=0.85\linewidth]{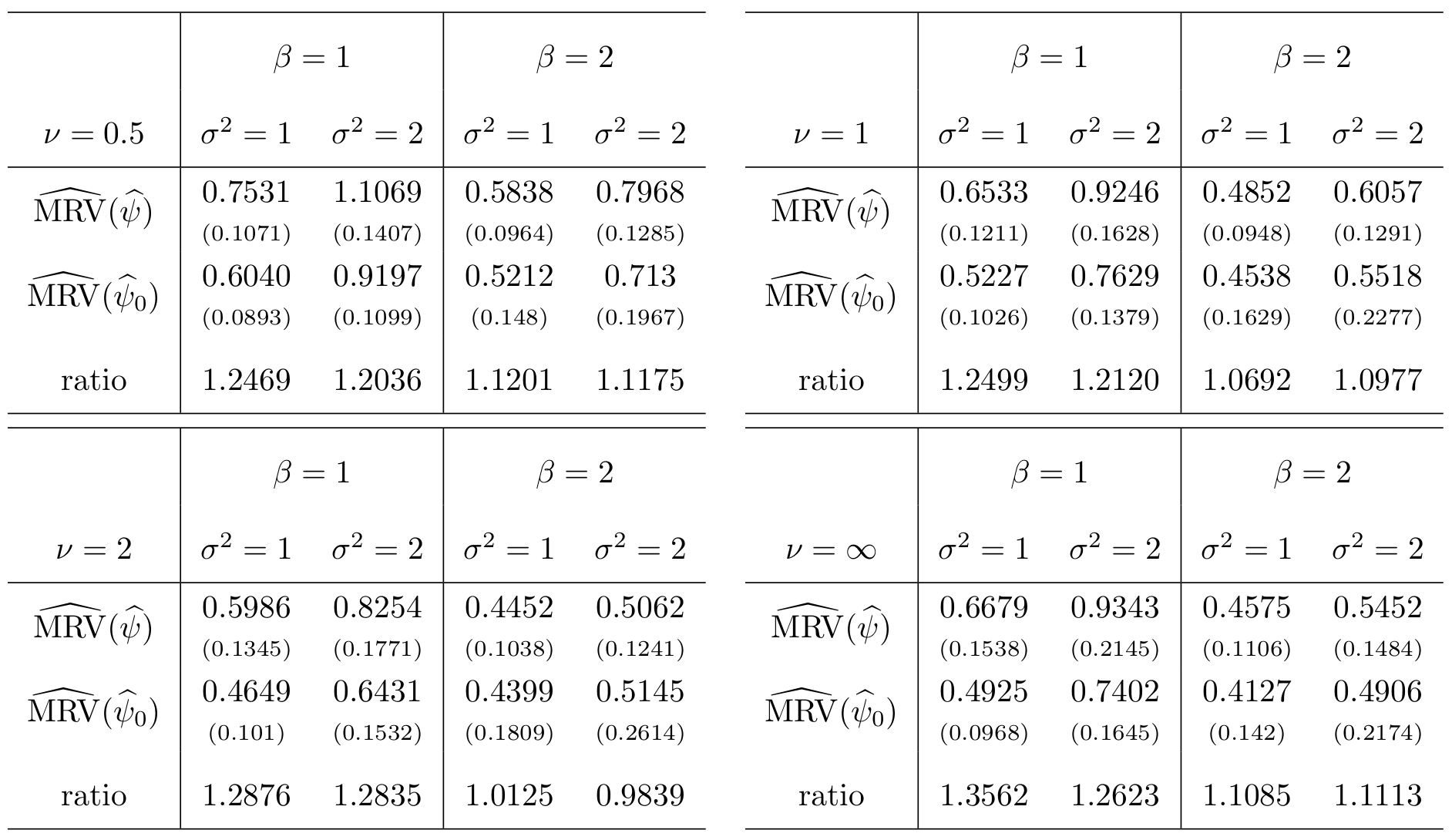}
	\caption{Results of the simulation study for the non-parametric estimators. The estimator $\widehat{\psi}$ is compared with its benchmark estimator $\widehat{\psi}_0$. The standard errors are  reported in brackets.}
	\label{fig:simunonpar}
\end{figure}

\vspace*{-5mm}
\begin{figure}[H]
	\centering
	\includegraphics[width=0.85\linewidth]{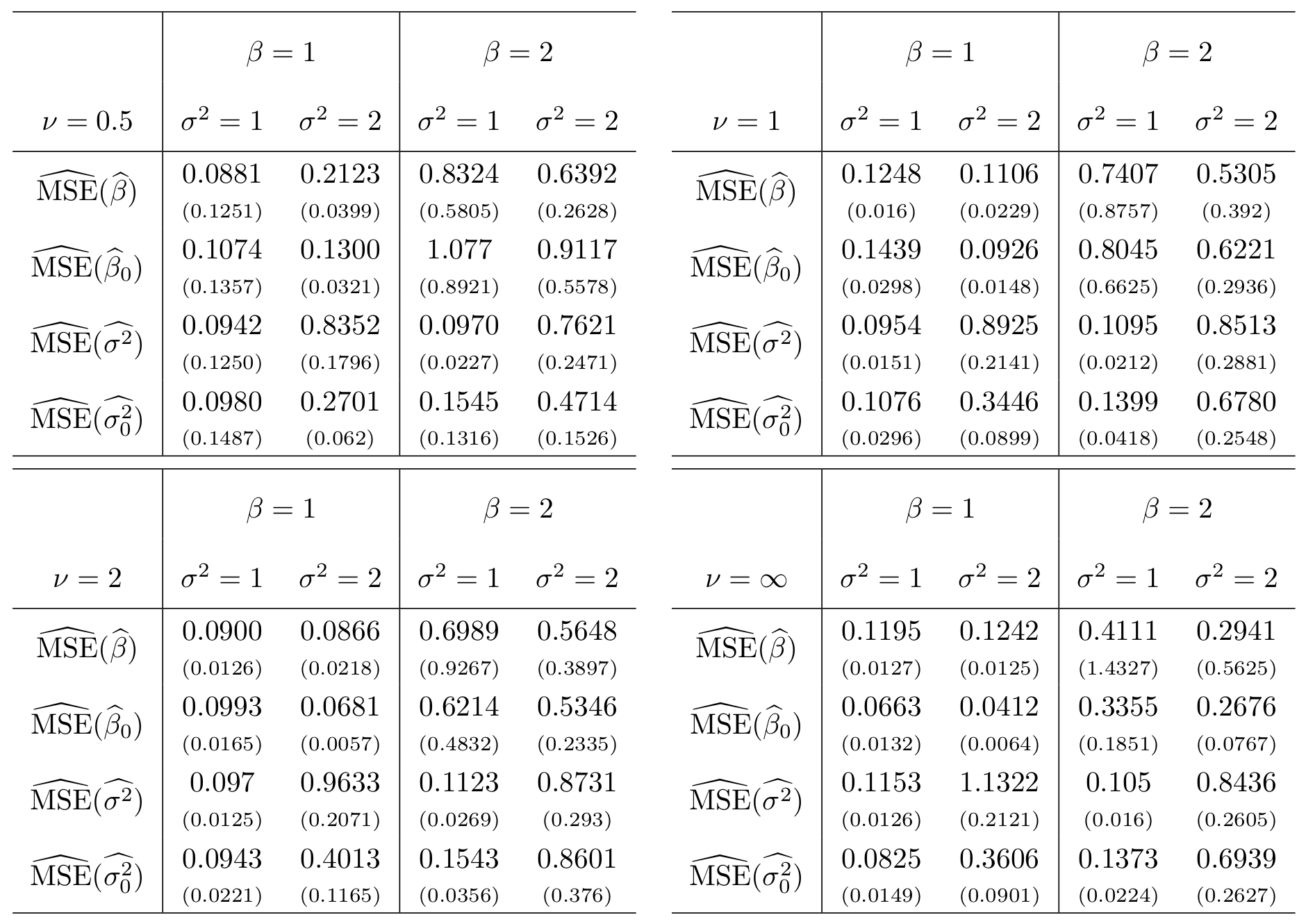}
	\caption{Results of the simulation study for the parametric estimators. The estimators $\widehat{\beta}$ and $\widehat{\sigma^2}$ are compared with their benchmark estimators $\widehat{\beta}_0$ and $\widehat{\sigma^2_0}$. The standard errors are  reported in brackets.}
	\label{fig:simupar}
	
\end{figure}

\newpage
\subsection{Preliminaries on point processes}\label{sec:pointprocesses} 
In this article, we follow the conventions based on \cite{daley2003pp, daley2008pp} briefly reviewed here.
Let $E$ be a complete separable metric space, and $\mathscr{B}(E)$ its Borel $\sigma$-field. A Borel measure $\mu$ on $E$ is called \textit{boundedly finite} if $\mu(A)<\infty$ for locally compact Borel sets $A$. 
The space of all \textit{boundedly finite} measures on $\mathscr{B}(E)$ is denoted by $\mathcal{M}$ and its subspace of \textit{simple counting measures} by $\mathcal{M}_p$. 
Both spaces, $\mathcal{M}$ and $\mathcal{M}_p$, are itself complete separable metric spaces if endowed with the weak hash topology, see Appendix 2.6 in \citep{daley2003pp}. Let $\mathscr{B}(\mathcal{M})$ and $\mathscr{B}(\mathcal{M}_p)$ be the smallest $\sigma$-algebra on $\mathcal{M}$ and accordingly $\mathcal{M}_p$, for which the mappings $\mu \rightarrow \mu(A)$ respectively $N\rightarrow N(A)$ are measurable for all $A$ in $\mathscr{B}(E)$.
A \textit{random measure} is a measurable mapping $\mu$ from a probability space $(\Omega,\mathscr{A},\P)$ to $(\mathcal{M},\mathscr{B}(\mathcal{M}))$ and accordingly a \textit{point process} is a  measurable mapping from $(\Omega,\mathscr{A},\P)$ to $(\mathcal{M}_p,\mathscr{B}(\mathcal{M}_p))$. Note that a point process, as defined here, always allows for a measurable enumeration, see Lemma~9.1.XIII in \cite{daley2008pp}. Convergence of random measures and point processes is always meant in the sense of weak convergence in $\mathcal{M}$ and $\mathcal{M}_p$, respectively. 

\subsection{Proofs}\label{sec:proofs}

\subsubsection{Proofs for Section~\ref{sec:model}}

\begin{proof}[Lemma~\ref{lemma:CP-CLT}]
	The point process on the left-hand side equals in distribution the $1/n$-thinning of $\sum_{i=1}^{n} N_i$ with $N_i \stackrel{i.i.d.}{\sim} CP(\Lambda_i)$. Hence, by Theorem~11.3.III in \cite{daley2008pp} the desired convergence holds true if and only if $n^{-1}\sum_{i=1}^{n} N_i$ converges to $\lambda$, which follows from the multivariate law of large numbers and Theorem~11.1.VII in \cite{daley2008pp}.
\end{proof}	

\begin{proof}[Lemma~\ref{lemma:sample-cts}]
	We follow closely the arguments of \citep[Proposition~13]{kabluchkoschlather2009br}.
	For $K \subset \R^d$ and $c>0$, set 
	\begin{align*}
		I_c(K)=\left\{i \in \N \,:\, \sup_{t \in K} u_iX_i(t-s_i)>c\right\}.
	\end{align*}
	\begin{enumerate}
		\item Conditional on the the process $\Psi$, the number of points in $I_c(K)$  is Poisson distributed with parameter
		\begin{align*} 
			\Lambda\left( \left\{(s,u,X) \,:\, \sup_{t \in K} uX(t-s)>c\right\} \right)
			=c^{-1} \, \E_X \int_{\R^d} \sup_{t \in K} X(t-s)\Psi(s)\d s, 
		\end{align*}	
		which is $\P_\Psi$-almost surely finite by the integrability condition \eqref{eq:int-cond-a}. Hence, the number of points in $I_c(K)$ is almost surely finite, which entails that 
		\begin{align*} 
			\sup_{t \in K}Y(t) \leq  \bigvee_{i \in I_{c}(K)} \sup_{t \in K} u_i X_i(t-s_i) \vee c
		\end{align*}	
		is almost surely finite. 
		\item Under the additional condition \eqref{eq:cond-sample-cts}, there exists $n \in \N$, such that 
		\begin{align*}
			Y(t) = \bigvee_{i \in I_{c}(K) \cup \{1,\dots,n\}} u_i X_i(t-s_i) \quad  \forall t \in K
		\end{align*}
		almost surely. That is, the process $Y$ can be represented on $K$ as the maximum of a finite number of continuous functions, which ensures the continuity of $Y$ on $K$.
	\end{enumerate}
\end{proof}

\begin{proof}[Lemma~\ref{lemma:sample-cts2}]
	Due to its compactness, $K$ can be split into finitely many (possibly overlapping) compact pieces $K_1, \dots, K_p$ of diameter less than $r$.
	Since $\inf_{s \in K} \Psi(s) >0$ almost surely, there are almost surely infinitely many elements in $I_j:=\{i \in \N \,:\, s_{i} \in K_j\}$ 
	of the Cox process $N$ (that underlies the construction of $Y$) in each of these pieces $K_j$, $j=1,\dots,p$. Since there exists an $r>0$ such that $\P_X( B_r(o)\subset \supp(X))>0$, there exists almost surely an element (in fact, infinitely many elements) $i_j \in I_j$ within the Cox process, such that $B_r(o) \subset \supp(X_{i_j})$.
	Summarizing, $K$ is almost surely covered by
	\begin{align*} 
		K \subset \bigcup_{j=1}^{p} K_j \subset \bigcup_{j=1}^{p} B_r(s_{i_j}) \subset \bigcup_{j=1}^{p} \supp(X_{i_j}(\cdot - s_{i_j})).
	\end{align*}
	Setting $n:=\max_{j=1}^p {i_j}$ and $c_j:=\inf_{s \in B_r(o)} X_{i_j}(s)>0$, we deduce
	that 
	\begin{align*}
		\inf_{t \in K} \bigvee_{i=1}^{n} u_iX_i(t-s_i) 
		\geq \inf_{t \in K} \bigvee_{j=1}^{p} u_{i_j} X_{i_j}(t - s_{i_j})
		\geq \inf_{t \in K} \bigvee_{j=1}^{p} u_{i_j} c_{j} \ind_{B_r(s_{i_j})}(t)
		\geq \bigvee_{j=1}^{p} u_{i_j} c_{j}>0
	\end{align*}
	is strictly greater than zero as desired.
\end{proof}

\begin{Remark}
	In fact, the condition $\inf_{s \in K} \Psi(s) >0$ in Lemma~\ref{lemma:sample-cts2} may be further relaxed to the requirement 
	\begin{align*}
		K \subset \underset{s\in \supp(\Psi)}{\bigcup} B_r(s).
	\end{align*}
\end{Remark}

\begin{proof}[Lemma~\ref{lemma: stationarity}]
	Due to the stationarity of $\Psi$ and the invariance of the Lebesgue measure with respect to translations, we obtain
	\begin{align*}
		&\P\Big(Y(t_1+h)\leq y_1, \dots, Y(t_k+h)\leq y_k\Big)\\
		&=\  \E_\Psi\Big[\exp\big(-\mu_Y^{-1}\E_X\int_{\R^d}\bigvee_{j=1}^k \frac{X(t_j+h-s)}{y_j}\Psi(s)~\d s\big)\Big]\\
		& =\  \E_\Psi\Big[\exp\big(-\mu_Y^{-1}\E_X\int_{\R^d}\bigvee_{j=1}^k \frac{X(t_j-s)}{y_j}\Psi(s+h)~\d s\big)\Big]\\
		&=\ \E_\Psi\Big[\exp\big(-\mu_Y^{-1}\E_X\int_{\R^d}\bigvee_{j=1}^k \frac{X(t_j-s)}{y_j}\Psi(s)~\d s\big)\Big]=\ \P\Big(Y(t_1)\leq y_1, \dots, Y(t_k)\leq y_k\Big).\\
	\end{align*}
\end{proof}

We will now prove Theorem~\ref{thm:domainofattraction}, the main result of Section~\ref{sec:model}, i.e.\ the weak convergence of the random fields  
\begin{align*}n^{-1}\left(\bigvee_{i=1}^nY_i\right)\rightarrow Z.\end{align*}
To this end we set the left-hand-side $Y^{(n)}:=n^{-1}\left(\bigvee_{i=1}^nY_i\right)$ which can be more conveniently represented as 
\begin{align*}
	Y^{(n)}(t)\stackrel{d}{=}\bigvee_{i=1}^\infty u_i X_i(t-s_i), \qquad t \in \R^d, 
\end{align*}	
where $N_n=\sum_{i=1}^\infty \delta_{(s_i,u_i,X_i)}$ is a Cox-process on $\R^d\times (0,\infty]\times \mathbb{X}$, directed by the random measure 
\begin{align*}
	\mathrm{d}\Lambda_n(s,u,X)=\mu_Y^{-1} \, n^{-1}\sum_{i=1}^n\Psi_i(s)\mathrm{d}s \, u^{-2} \mathrm{d}u \, \mathrm{d}\P_X,
\end{align*}
and $\Psi_i$, $i=1, \dots, n$ represent i.i.d.\ copies of $\Psi$. The random measure $\Lambda_n$ is the directing measure of the union of the underlying independent Cox-processes of the random fields $Y_i$, $i=1,\dots,n$, scaled by $n^{-1}$. By Lemma~\ref{lemma:CP-CLT}, the point process $N_n$ converges weakly to the Poisson-process with directing measure
\begin{align*}
	\d\lambda(s,u,X)= {\mu_Z}^{-1} \, \d s \, u^{-2} \d u \, \d \P_X
\end{align*}
that underlies the mixed moving maxima random field $Z$. The latter convergence indicates already the result of Theorem~\ref{thm:domainofattraction}.
In order to prove Theorem~\ref{thm:domainofattraction}, we show first the convergence of the finite dimensional distributions and then the tightness of the sequence $Y^{(n)}$, $n=1,2,\dots$.

\begin{Lemma}[Convergence of finite-dimensional distributions] \label{lemma:fidi}
	Let the random fields $Y$ and $Z$ be specified as in Theorem~\ref{thm:domainofattraction}, then the finite dimensional distributions of $Y^{(n)}$ converge to those of $Z$ as $n \to \infty$.	
\end{Lemma}

\begin{proof} We fix $t_1,\dots,t_k \in \R^d$ and show that the random vector $(Y(t_1),\dots,Y(t_k))$ lies in the max-domain of attraction of the random vector $(Z(t_1),\dots,Z(t_k))$. It then automatically follows that the finite dimensional distributions $Y^{(n)}$ converge to those of $Z$, since the scaling constants for each individual $t \in \R^d$ are chosen appropriately. 
	
	For $y=(y_1,\dots,y_k) \in (0,\infty)^d$, it follows from \eqref{eq:int-cond-b} that the non-negative random variable 
	\begin{align*}
		H_\Psi(y):= \mu_Y^{-1}\E_X\int\underset{1\leq j\leq k}{\max}\frac{X(t_j-s)}{y_j} \Psi(s)~\d s
	\end{align*}
	satisfies that its first moment $\E_\Psi(H_\Psi(y))<\infty$ is finite and can be gained from its Laplace transform via 
	\begin{align*}
		\E_\Psi(H_\Psi(y))= - \lim_{t \downarrow 0} \,\frac{d}{\mathrm{d}t}\, \E_\Psi \left(e^{- t \, H_\Psi(y)}\right). 
	\end{align*}
	Hence, by l'H{\^o}pital's rule
	\begin{align*}
		&\lim_{\lambda \to \infty}
		\frac{1-\P(Y(t_1) \leq \lambda y_1, \dots, Y(t_k) \leq \lambda y_k)}{1-\P(Y(t_1) \leq \lambda , \dots, Y(t_k) \leq \lambda )}
		= \lim_{t \to 0}
		\frac{1-\E_{\Psi}(e^{- t \, H_\Psi(y)})}{1-\E_{\Psi}(e^{- t \, H_\Psi(1)})}\\
		&= \frac{\E_\Psi(H_\Psi(y))}{\E_\Psi(H_\Psi(1))}=:V(y)
	\end{align*}
	with $V(cy)=c^{-1}y$. Moreover, $V(y)$ is a multiple of the exponent function of the max-stable random vector $(Z(t_1),\dots,Z(t_k))$
	\begin{align*}
		- \log \P(Z(t_1) \leq y_1, \dots, Z(t_k) \leq y_k) 
		= \mu_Z^{-1}\E_X\int\underset{1\leq j\leq k}{\max}\frac{X(t_j-s)}{y_j}~\d s
		= \E_\Psi(H_\Psi(y)).
	\end{align*}
	Hence, by \cite[Corollary~5.18 (a)]{resnick2008neuauflage}, the random vector 
	$(Y(t_1),\dots,Y(t_k))$ lies in its domain of attraction.
\end{proof}







The following lemma will be useful to prove the tightness of the sequence $Y^{(n)}$.

\begin{Lemma}\label{lemma:ineq}
	Let $a_n$ and $b_n$ be bounded sequences of non-negative real numbers, then 
	\begin{align*}
		\left|\bigvee_{n=1}^{\infty}a_n-\bigvee_{n=1}^{\infty}b_n\right|\leq \bigvee_{n=1}^{\infty}\left| a_n-b_n\right|.
	\end{align*}
\end{Lemma}
\begin{proof}
	The statement is the triangle inequality $\lvert \lVert a \rVert_{\infty}-\lVert b \rVert_{\infty} \rvert \leq \lVert a-b\rVert_{\infty}$ with $\|\cdot\|_{\infty}$ the $\ell^{\infty}$ norm in the space of bounded sequences.
\end{proof}

\begin{Lemma}[Tightness]\label{lemma:tightness}
	Let the random field $Y$ be specified as in Theorem~\ref{thm:domainofattraction}, then the sequence of random fields $Y^{(n)}$ is tight. 
\end{Lemma} 

\begin{proof}
	Since the finiteness of $Y$ does also ensure the finiteness of each $Y^{(n)}$, it suffices to show that, for a compact set $K \subset \R^d$, the modulus of continuity 
	\begin{align*}
		\omega_K\left(Y^{(n)},\delta\right):=\underset{t_1,t_2\in K \,:\, \|t_1-t_2\|\leq\delta}{\sup}\left|Y^{(n)}(t_1)-Y^{(n)}(t_2)\right|
	\end{align*}
	satisfies the convergence
	\begin{align}\label{eq:tight}
		\underset{\delta\rightarrow 0}{\lim}\ \underset{n\rightarrow \infty}{\limsup}\ \P\left(\omega_K\left(Y^{(n)},\delta\right)> \varepsilon\right)=0.
	\end{align}
	To simplify the notation, we introduce 
	$K_{\delta}:=\left\{(t_1,t_2) \in \R^d \times \R^d \,:\, \|t_1-t_2\|\leq\delta, t_1,t_2\in K\right\}$.
	By the definition of $Y^{(n)}$ and the preceding Lemma~\ref{lemma:ineq}, we have 
	\begin{align*}
		\P\left( \omega_K\left(Y^{(n)},\delta\right) \leq \varepsilon\right)
		&=\  \P\left(\underset{(t_1,t_2)\in K_{\delta}}{\sup}\left|\bigvee_{i=1}^{\infty} u_iX_i(t_1-s_i)-\bigvee_{i=1}^\infty u_iX_i(t_2-s_i)\right| \leq \varepsilon\right)\\
		&\geq\  \P\left(\underset{(t_1,t_2)\in K_{\delta}}{\sup}\bigvee_{i=1}^{\infty}u_i\left|X_i(t_1-s_i)-X_i(t_2-s_i)\right| \leq \varepsilon\right).
	\end{align*}
	As the tuples $(s_i,u_i,X_i)$, $i\in\N,$ are the points of the Cox process $N_n$, we can compute the latter probability as expected void-probability. 
	To this end, let us denote the joint probability law of the i.i.d.\ intensity processes $\Psi_i$, $i=1,2,\dots$ and its expectation by $\P_\Psi$ and $\E_\Psi$, respectively. Setting $\Psi^{(n)}(s):=n^{-1}\sum_{i=1}^n\Psi_i(s)$, we obtain
	\begin{align*}
		&\liminf_{n \to \infty} \P\left( \omega_K\left(Y^{(n)},\delta\right) \leq \varepsilon\right)\\
		\geq\ & \liminf_{n\to \infty}\E_{\Psi}\left[\exp\left(-\varepsilon^{-1}\mu_Y^{-1}\E_X\int_{\R^d}\ \underset{(t_1,t_2)\in K_{\delta}}{\sup}\left|X(t_1-s)-X(t_2-s)\right|\Psi^{(n)}(s)\d s \right)\right]\\
		\geq\ & \E_{\Psi}\left[ \liminf_{n\to \infty} \exp\left(-\varepsilon^{-1}\mu_Y^{-1}\E_X\int_{\R^d}\ \underset{(t_1,t_2)\in K_{\delta}}{\sup}\left|X(t_1-s)-X(t_2-s)\right|\Psi^{(n)}(s)\d s \right)\right],
	\end{align*}	
	where the last inequality follows from Fatou's Lemma.
	Moreover, the strong law of large numbers and condition \eqref{eq:int-cond-a} (which ensures the existence and finiteness of the following right-hand side) yield that $\P_{\Psi}$-almost surely
	\begin{align*}
		&\lim_{n \to \infty} \E_X\int_{\R^d} \underset{(t_1,t_2)\in K_{\delta}}{\sup}  \left|X(t_1-s)-X(t_2-s)\right|\Psi^{(n)}(s)\d s\\
		&= \ c_{\Psi}\ \E_X\int_{\R^d} \underset{(t_1,t_2)\in K_{\delta}}{\sup} \left|X(t_1-s)-X(t_2-s)\right|\d s\\
		&\leq \ c_{\Psi}\ \E_X\int_{\R^d}  \underset{t \in B_{\delta}(o)}{\sup}   \left|X(s-t)-X(s)\right|\d s,
	\end{align*}
	which entails
	\begin{align*}
		\liminf_{n \to \infty} \P\left( \omega_K\left(Y^{(n)},\delta\right) \leq \varepsilon\right)
		\geq\  
		\exp\left(-\varepsilon^{-1}\mu_Y^{-1} c_\Psi\ \E_X\int_{\R^d}\ \underset{t \in B_{\delta}(o)}{\sup}\left|X(s-t)-X(s)\right|\d s \right).
	\end{align*}
	Finally, in order to establish \eqref{eq:tight}, it remains to be shown that 
	\begin{align*}
		\lim_{\delta \to 0}\ \E_X\int_{\R^d}\ \underset{t \in B_{\delta}(o)}{\sup}\left|X(s-t)-X(s)\right|\d s = 0.
	\end{align*}
	This, however, follows from the dominated convergence theorem, since for any fixed \\ $X\in C(\R^d)$ and any fixed $s \in \R^d$ the convergence of the integrand to $0$ holds true and by 
	\begin{align*}
		\underset{t \in B_{\delta}(o)}{\sup}\left|X(s-t)-X(s)\right|
		\leq \underset{t \in B_{\delta}(o)}{\sup} X(s-t) + X(s)
	\end{align*}
	and condition \eqref{eq:int-cond-X}, there exists an integrable upper bound.
\end{proof}

We are now in position to prove the main result of Section~\ref{sec:model}.

\begin{proof}[Theorem~\ref{thm:domainofattraction}]
	The finiteness and sample-continuity of the random fields $Y$ and $Z$ are an immediate consequence of Lemma~\ref{lemma:sample-cts}. While Lemma~\ref{lemma:fidi} shows that the finite-dimensional distributions of the random fields $Y^{(n)}$ converge to those of the process $Z$, Lemma~\ref{lemma:tightness} establishes the tightness of the sequence $Y^{(n)}$. Collectively, this proves the assertions.
\end{proof}

\subsubsection{Proofs for Section~\ref{sec:simu}}

\begin{proof}[Proposition \ref{Prop:Simu}]
	First note that $\sum_{i=1}^{\infty}\delta_{\Gamma_i}$ is a Poisson process on $\R_+$ with intensity 1. Hence $\sum_{i=1}^{\infty}\delta_{\Gamma^{-1}_i}$ is a Poisson process on $(0,\infty]$ with intensity $u^{-2}du$.
	Attaching the independent markings $X_i$ and, for fixed $\Psi=\psi$, the markings $S_i \sim \psi(s)/\nu(D_R)$ yields that, for fixed $\Psi=\psi$, the point process $\sum_{i=1}^\infty \delta_{(S_i,\nu_\psi(D_R)\mu_Y^{-1}\Gamma_i^{-1},X_i)}$  is a Poisson process directed by the measure $\mu_Y^{-1}\psi(s)~\d s\,u^{-2}~\d u~\d\P_X$ on $D_R\times[0,\infty)\times \mathbb{X}$. 
	
	Since in the construction of $Y$ only storms with center in $D_R$ can contribute to the process $Y$ on $D$, the law of $Y$ on $D$ and the law of 
	\begin{align*}
		\frac{\nu_\psi(D_R)}{\mu_Y} \bigvee_{i=1}^{\infty} \Gamma_i^{-1}X_i(t-S_i),\quad t\in D
	\end{align*}
	coincide. By definition of the stopping time $T$ and since $X$ is uniformly bounded by $C$, the latter has the same law as the process $\widetilde{Y}$ on $D$.
	So, it remains to be shown that $T$ is almost surely finite.
	
	Similar to the proof of Lemma~\ref{lemma:sample-cts2}, it can be shown that 
	\begin{align}
		\exists\, n \in \N \,:\, \inf_{t \in D} \bigvee_{i=1}^n \Gamma^{-1}_i X_i(t-S_i) > 0 \quad \text{almost surely}.
	\end{align}	
	Together with the decrease of the sequence $\Gamma_{n+1}^{-1}$
	this implies the a.s.-finiteness of $T$.
\end{proof}

\subsubsection{Proofs for Section~\ref{sec:estpsinonpara}}

%

\begin{proof}[Lemma~\ref{lemma:contpointsint}]
	
	Since $Y^{\ast}\vert_{\Psi=\psi}$ is independent of $N|_{\Psi=\psi}$, the process $N_K^{Y^{\ast}}\vert_{\Psi=\psi,Y^{\ast}=y}$ is an independent thinning of the Poisson process $N\vert_{\Psi=\psi}$. The number of points in the set
	\begin{align*}
		\left\{s\in K_R: (s,u,X)\in N\vert_{\Psi=\psi},\ u^{-1}\leq\sup_{t\in K} \frac{X(t-s)}{y(t)}\right\}
	\end{align*}
	is Poisson distributed with parameter
	\begin{align*}
		{\mu_Y}^{-1}\int_{K_R}\E_X\sup_{t\in K}\ \frac{X(t-s)}{y(t)}\psi(s)~\d s.
	\end{align*}
	
	This finishes the proof.	
\end{proof}

\begin{proof}[Theorem~\ref{lemma:uniformconvergence}]
	The paths $\psi$ of the intensity process $\Psi$ are almost surely continuous and $\psi_K^y$ is even uniformly continuous since its support equals the compact set $K_R$. Therefore, the integral	$\int\|s\|\psi^y_K(s)\d s<\infty$ is finite. Hence all assumptions of Theorem 3.1.7. in \cite{rao83unifconv} hold true, which proofs the statement.
\end{proof}
\begin{proof}[Corollary \ref{cor:uniformconvergence}]
	Condition \eqref{cond:Xboundary} ensures the existence of an upper bound $C\in\R$ such that
	\begin{align*}
		0<b^{y}_K(s)\leq C<\infty,\quad \forall s\in K_{R-\varepsilon},
	\end{align*}
	almost surely. Combining this with Theorem~\ref{lemma:uniformconvergence} we obtain
	\begin{align*}
		\underset{s\in K_{R-\varepsilon}}{\sup\ }|\widehat{{\psi}}_{K,n}(s)-{\psi}(s)|\leq C \cdot \underset{s\in K_R}{\sup\ }|\widehat{{\psi}}^{y}_{K,n}(s)-{\psi}^{y}(s)| \rightarrow 0,
	\end{align*}
	almost surely.
\end{proof}

\begin{proof}[Lemma~\ref{lemma:unbiased}]
	The assertion follows from the straight forward computation
	\begin{align*}
		&\E\int_{D} \widehat{{\psi}^y_{h}}(s)~\d s=\E\int_{D}h^{-d}\sum_{t\in N_K^y\cap D} c_{D}(t)^{-1} k\left(\frac{s-t}{h}\right)~\d s=\E\sum_{t\in N_K^y\cap D} 1=\E N_K^y(D).
	\end{align*}
	Since the number of points in $N_K^y(D) |_{\Psi=\psi}$ is Poisson distributed with parameter
	\begin{align*}
		\mu_Y^{-1} \, 	\int_{D}\E_X\sup_{t\in K}\ X(t-s)y(t)^{-1}\psi(s)~\d s,
	\end{align*}
	we conclude
	\begin{align*}
		\E\int_{D} \widehat{{\psi}^y_{h}}(s)~\d s=\mu_Y^{-1}\E\left[\int_{D}\E_X\ \left(\underset{t\in K}{\sup\ }{X(t-s)}{y(t)}^{-1}\right)\psi(s)\d s\right]=\E\int_{D}{\psi^y_K}(s)~\d s.
	\end{align*}	
\end{proof}

\begin{proof}[Lemma~\ref{lemma: postprocessing}]
	Simple calculation shows that $p\cdot CP(f\Psi)=p\cdot CP(f\Psi)\vert_{\{f\geq 1\}}+CP(f\Psi)\vert_{\{f< 1\}}$ and\\ $(1-f)_+\Psi=(1-f)\Psi\vert_{\{f< 1\}}$ which implies			
	\begin{align*}
		\notag &p\cdot CP(f\Psi)+CP((1-f)_+\Psi)\\
		&=p\cdot CP(f\Psi)\vert_{\{f\geq 1\}}+ CP(f\Psi)\vert_{\{f< 1\}}+CP((1-f)\Psi)\vert_{\{f< 1\}}.
	\end{align*}
	Since $p=f^{-1}\cdot\1_{\{f\geq 1\}}+\1_{\{f<1\}}$ we obtain $$p\cdot CP(f\Psi)\vert_{\{f\geq 1\}}=f^{-1}\cdot CP(f\Psi)\vert_{\{f\geq 1\}}=CP(\Psi)\1_{\{f\geq 1\}}$$ for the first part of the sum on the set $\{f\geq 1\}$.
	Furthermore, the remaining parts satisfy $CP(f\Psi)\vert_{\{f< 1\}}+CP((1-f)\Psi)\vert_{\{f< 1\}}=CP(\Psi)\vert_{\{f< 1\}}$ on the set $\{f<1\}$ which entails the assertion~\eqref{eq:CPcomp}.
\end{proof}

\bibliography{Literatur}
\bibliographystyle{plainnat}

\end{document}